\pdfoutput=1
\RequirePackage{ifpdf}
\ifpdf 
\documentclass[pdftex]{sigma}
\else
\documentclass{sigma}
\fi

\usepackage{mathrsfs}\usepackage[all]{xy}

\renewcommand\a{\alpha}
\renewcommand\b{\beta}
\newcommand\g{\gamma}
\renewcommand\d{\delta}
\newcommand\e{\varepsilon}
\renewcommand\th{\theta}

\newcommand\ol{\overline}

\newcommand{\wh}{\widehat}

\numberwithin{equation}{section}

\newtheorem{Theorem}{Theorem}[section]

\newtheorem{Lemma}[Theorem]{Lemma}
\newtheorem{Proposition}[Theorem]{Proposition}
 { \theoremstyle{definition}
\newtheorem{Definition}[Theorem]{Definition}
\newtheorem*{Notation}{Notation}

\newtheorem{Remark}[Theorem]{Remark} }

\begin{document}

\allowdisplaybreaks

\newcommand{\arXivNumber}{2411.18961}

\renewcommand{\PaperNumber}{090}

\FirstPageHeading

\ShortArticleName{The Fefferman Metric for Twistor CR Manifolds}

\ArticleName{The Fefferman Metric for Twistor CR Manifolds\\ and Conformal Geodesics in Dimension Three}

\Author{Taiji MARUGAME}

\AuthorNameForHeading{T.~Marugame}

\Address{Department of Mathematics, The University of Electro-Communications, \\
1-5-1 Chofugaoka, Chofu, Tokyo 182-8585, Japan}
\Email{\href{mailto:marugame@uec.ac.jp}{marugame@uec.ac.jp}}

\ArticleDates{Received June 03, 2025, in final form October 16, 2025; Published online October 24, 2025}

\Abstract{We give an explicit description of the Fefferman metric for twistor CR manifolds in terms of Riemannian structures on the base conformal $3$-manifolds. As an application, we prove that chains and null chains on twistor CR manifolds project to conformal geodesics, and that any conformal geodesic has lifts both to a chain and a null chain. By using this correspondence, we give a variational characterization of conformal geodesics in dimension three which involves the total torsion functional.}


\Classification{53B20; 32V05; 53C18}

\section{Introduction}
The Fefferman construction in CR geometry associates a pseudo-Riemannian conformal manifold $\bigl(\mathcal{C}, \bigl[g^{\mathrm{F}}\bigr]\bigr)$, called the {\it Fefferman space}, to any non-degenerate CR manifold ${\bigl(M, T^{1, 0}M\bigr)}$~\cite{F, Lee}.
 The manifold $\mathcal{C}$ is a circle bundle over $M$ and the Fefferman metric $g^{\mathrm{F}}$ is an $S^1$-invariant metric on $\mathcal{C}$ which is conformally flat if and only if $M$ is CR isomorphic to the hyperquadric, i.e., the flat model of CR manifold. Since the conformal structure $\bigl[g^{\mathrm{F}}\bigr]$ is canonically defined from the CR structure, the construction enables us to apply conformal geometry methods to the study of CR manifolds.

 The aim of this paper is to describe explicitly the Fefferman metric for a special class of~CR $5$-manifolds, namely {\it twistor CR manifolds}, and apply it to $3$-dimensional conformal geometry. Given a conformal $3$-manifold $\bigl(\Sigma^3, [g]\bigr)$ of Riemannian signature, LeBrun's twistor CR manifold~\cite{LeB} is defined as the space of projectivized complex null (co)vectors:
\[
M^5:=\{[\zeta]\in\mathbb{P}(\mathbb{C}T^*\Sigma)\mid g(\zeta, \zeta)=0 \}.
\]
Each fiber of $M$ is a rational curve in $\mathbb{P}(\mathbb{C}T_x^*\Sigma)\cong \mathbb{CP}^2$, and it is shown that $M$ has a canonical CR structure of Lorentzian signature. We mention the work~\cite{Low} of Low, which also constructs the Fefferman metric of twistor CR manifolds over more general 3-manifolds with connections. Our description of the Fefferman metric is different from that of~\cite{Low}, and has the advantage that it is constructed more directly from the conformal structure on $\Sigma$ without passing through the Tanaka--Webster connection on $M$.

Fixing a representative metric $g\in[g]$, we show that the $\mathrm{U}(1)$-bundle
\[
\wh{\mathcal{C}}:=\bigl\{\zeta\in\mathbb{C}T^*\Sigma\mid g(\zeta, \zeta)=0,\, g\bigl(\zeta, \ol\zeta\bigr)=1\bigr\}
\]
over $M$ can be identified with a double covering of the Fefferman space $\mathcal{C}$ by constructing an `adapted' coframe \smash{$\bigl(\tilde\theta^0=\theta, \tilde\theta^1, \tilde\theta^2, \tilde\theta^{\ol1}, \tilde\theta^{\ol2}, \tilde\theta^3\bigr)$} on \smash{$\wh{\mathcal{C}}$} out of the Riemannian structure $g$.
Here, $\th$ is a~contact form on $M$ determined by $g$, and \smash{$\bigl(\tilde\th^1, \tilde\th^2\bigr)$} is a `twisted' coframe for $T^{1, 0}M$ while~\smash{$\tilde\th^3$} restricts to a coframe for the $S^1$-fibers. In this coframe, (the pullback of) the Fefferman metric~$g^{\mathrm{F}}$ on \smash{$\wh{\mathcal{C}}$} is written as
\[
g^{\mathrm{F}}=2\bigl(\tilde\theta^1\cdot\tilde\theta^{\ol2}+\tilde\theta^2\cdot\tilde\theta^{\ol1}+\theta\cdot \tilde\theta^3\bigr).
\]
Using this expression, we compute the Levi-Civita connection and the curvature tensor of $g^{\mathrm{F}}$ in terms of those of $g$. In particular, the Weyl curvature of $g^{\mathrm{F}}$ is described by the Hodge dual of the Cotton tensor of $g$. As a consequence, we can recover \cite[Theorem~6.7]{Low}, which states that the Fefferman metric is conformally flat if and only if $(\Sigma, [g])$ is conformally flat (Theorem~\ref{flat}).

As another application, we establish a correspondence between distinguished curves on the CR manifold $\bigl(M, T^{1, 0}M\bigr)$ and the conformal manifold $(\Sigma, [g])$.

Since null geodesics of $g^{\mathrm{F}}$ are conformally invariant up to reparametrizations, their projections define a CR invariant family of curves on $M$, called {\it $($null$)$ chains}. The infinitesimal generator~$K$ of the $S^1$-action on $\wh{\mathcal{C}}$ is a Killing vector field,
so $c:=g^{\mathrm{F}}(K, \dot \gamma(t))$ is constant in $t$ for any null geodesic $\gamma$. The projection of $\gamma$ is called a chain if $c\neq0$ and a null chain if $c=0$. Chains are transverse to the contact distribution $H=\operatorname{Re} T^{1, 0}M$ and uniquely determined by the initial velocity. Thus, they can be considered as `geodesics' in CR geometry. In contrast, null chains are tangent to $H$ and we also need information of the initial acceleration to specify a null chain; see~\cite{Koch}.

The base conformal manifold $(\Sigma, [g])$, on the other hand, has a conformally invariant family of curves, called {\it conformal geodesics} or {\it conformal circles}. They are defined by a conformally invariant
third order ODE
\[
\nabla_{\dot x}\nabla_{\dot x}\dot x-\frac{3g(\dot x, \nabla_{\dot x}\dot x)}{|\dot x|^2}\nabla_{\dot x}\dot x+\frac{3|\nabla_{\dot x}\dot x|^2}{2|\dot x|^2}\dot x
+2P(\dot x, \dot x)\dot x-|\dot x|^2P(\dot x)=0,
\]
where $P$ denotes the Schouten tensor of $g\in[g]$.
Due to its conformal invariance, conformal geodesics play an important role in general relativity to examine the structure of spacetime; see, e.g.,~\cite{DT, FS}.
Although the conformal geodesic equation specifies distinguished parametrizations of $x(t)$, called projective parameters \cite{BE}, we consider conformal geodesics as unparametrized curves in this paper.

By computing the null geodesic equation for the Fefferman metric $g^{\mathrm{F}}$, we prove that these invariant curves on $M$ and $\Sigma$ are related as follows.

\begin{Theorem}\label{chain-projection}
Let $(\Sigma, [g])$ be a conformal $3$-manifold and $M$ the twistor CR manifold over $\Sigma$. Then, the following hold:
\begin{itemize}\itemsep=0pt
\item[$(1)$] Any chain and any null chain on $M$ project to a constant curve or a conformal geodesic on $\Sigma$.
\item[$(2)$] For any conformal geodesic $x(t)$, $a\le t\le b$, on $\Sigma$ and $[\zeta_0]\in M_{x(a)}$, there exists a unique chain $($resp.\ a null chain$)$ through $[\zeta_0]$ which projects to $x(t)$ when $g(\eta_0, \dot x(a))\neq0$ $($resp.\ $g(\eta_0, \dot x(a))=0)$, where $\eta_0:= {\rm i}\zeta_0\times \ol{\zeta_0}$.
\end{itemize}
\end{Theorem}

This theorem asserts that any (null) chain projects to a conformal geodesic, and conversely we can lift any conformal geodesic to a (null) chain. The lift is determined by a choice of the lift of the endpoint $x(a)$, but in fact there is a canonical choice, namely the lift with $\eta_0$ being proportional to the initial velocity $\dot x(a)$. The lift is described as follows:
We fix a metric $g\in[g]$ and a (local) orientation of $\Sigma$. For any regular curve $x(t)$ on $\Sigma$, we have a complex null (co)vector field $\zeta(t)$, unique up to the $\mathrm{U}(1)$-actions, such that $\bigl(\sqrt{2}\operatorname{Re}\zeta, \sqrt{2}\operatorname{Im}\zeta\bigr)$ gives an~oriented orthonormal basis of $\dot x(t)^\perp\subset T_{x(t)}\Sigma$. The projective class $[\zeta(t)]$ is independent of the choice of $g\in[g]$ and we have a canonical lift $\tilde x(t):=(x(t), [\zeta(t)])$ on $M$, which is transverse to the contact distribution. Then, we can prove that $\tilde x(t)$ is a chain if and only if $x(t)$ is a~conformal geodesic (Theorem~\ref{canonical-lift}).

There is an alternative description of the Fefferman space $\bigl(\wh{\mathcal{C}}, g^{\mathrm{F}}\bigr)$ and the canonical lift $\tilde x(t)$ in terms of the unit tangent sphere bundle $S\Sigma\subset T\Sigma$ and the oriented orthonormal frame bundle~$\mathcal{P}$ over $(\Sigma, g)$: Their relations are given by the bundle isomorphisms
\[
\xymatrix@R=10pt{
\wh{\mathcal{C}} \ar[r]^{\sim} \ar[d] & \mathcal{P} \ar[d] & \zeta \longmapsto \bigl(\sqrt{2}\operatorname{Re}\zeta, \sqrt{2}\operatorname{Im}\zeta, \eta\bigr), \\
M \ar[r]^{\sim} & S\Sigma, & \hspace{-3cm} [\zeta]\longmapsto \eta,
}
\]
where we define $\eta:={\rm i}\zeta\times \ol\zeta=2\operatorname{Re}\zeta\times \operatorname{Im}\zeta$ with the cross product with respect to the representative metric $g$.
By using the canonical $1$-form and the principal connection $1$-form
\[
\phi=\begin{pmatrix}
\phi^1 \\ \phi^2 \\ \phi^3
\end{pmatrix}
, \qquad \omega=\begin{pmatrix}
\hphantom{-}0 & \hphantom{-}\omega^3 & -\omega^2 \\
-\omega^3 & \hphantom{-}0 & \hphantom{-}\omega^1 \\
\hphantom{-}\omega^2 & -\omega^1 & \hphantom{-}0
\end{pmatrix}
\]
on $\mathcal{P}$, the Fefferman metric $G^{\mathrm{F}}$ on $\mathcal{P}$ corresponding to $g^{\mathrm{F}}$ is represented as
\[
G^{\mathrm{F}}=2\bigl( \phi^1\cdot\omega^1+ \phi^2\cdot\omega^2+ \phi^3\cdot\omega^3\bigr).
\]
The equivariance of $\phi$, $\omega$ shows that $G^{\mathrm{F}}$ is a right $\mathrm{SO}(3)$-invariant metric. For a regular curve~$x(t)$ on $\Sigma$, the canonical lift $\tilde x(t)$ on $S\Sigma$ is given by the jet lift
\begin{equation}\label{jet-lift}
\tilde x(t)=\biggl(x(t), \frac{\dot x(t)}{|\dot x(t)|}\biggr).
\end{equation}

We can use this lift to derive a variational principle for conformal geodesics in dimension three. Recently, Dunajski--Kry\'nski~\cite{DK} gave a variational characterization of conformal geodesics (in general dimensions) as the critical curves of an integral functional under a restricted class of variations. In this paper, we give another characterization which is specific to the $3$-dimensional~case.

In~\cite{CMMM}, it is proved that chains on CR manifolds coincide with geodesics of a
certain Kropina-type singular Finsler metric obtained by applying Fermat's principle to the Fefferman metric. In our setting, the Kropina metric on $S\Sigma$ is given by
\[
F(\tilde x, \xi):=\frac{G^{\mathrm{F}}(s_* \xi, s_*\xi)}{G^{\mathrm{F}}(K, s_*\xi)}, \qquad \tilde x\in U\subset S\Sigma,\quad \xi\in T_{\tilde x}S\Sigma\setminus\operatorname{Ker}\th,
\]
for each choice of a local section $s\colon U\to \mathcal{P}$. The Kropina metric is not defined along the contact distribution since $K\lrcorner\, G^{\mathrm{F}}$ descends to the contact form $\th$. Using the canonical lift \eqref{jet-lift}, we define a functional $\mathscr{L}$ on the space of regular curves on $\Sigma$ by
\[
\mathscr{L}[x(t)]:=\int_a^b F\bigl(\tilde x(t), \dot{\tilde x}(t)\bigr)\, {\rm d}t
\biggl(=2\int_a^b \omega^3\bigl(s_*\dot{\tilde x}\bigr)\,{\rm d}t\biggr),
\]
where $s$ is a local section of $\mathcal{P}\to S\Sigma$ on a neighborhood of the curve $\tilde x(t)$.
If $x(t)$ is a conformal geodesic, then $\tilde x(t)$ becomes a chain and hence $x(t)$ is a critical curve of $\mathscr{L}$ under the variations fixing the lifts $\tilde x(a)$, $\tilde x(b)$ of the endpoints.
Conversely, if $x(t)$ is a critical curve of $\mathscr{L}$, the Kropina length functional is stationary at $\tilde x(t)$ with respect to the variations given by the lift of variations of $x(t)$. Combined with the fact that the variations of the Kropina length in the vertical directions always vanish (Proposition~\ref{vertical-variation}), this implies that $\tilde x(t)$ becomes a~chain. Consequently, conformal geodesics are characterized as the critical curves of $\mathscr{L}$ under the variations fixing $\tilde x(a)$, $\tilde x(b)$ (Theorem~\ref{variational-characterization}).

The functional $\mathscr{L}$ is closely related to the so-called {\it total torsion functional}, which we denote by~$\mathscr{T}$. For a regular curve $x(t)$, $a\le t\le b$, on $(\Sigma, g)$, we fix an isometry $\dot x(a)^\perp\cong\dot x(b)^\perp$ by choosing oriented orthonormal basis $A=(e_1, e_2, \dot x(a)/|\dot x(a)|)\in \mathcal{P}_{x(a)}$ and $B=\bigl(e'_1, e'_2, \dot x(b)/|\dot x(b)|\bigr)\in \mathcal{P}_{x(b)}$ at the endpoints. Then, $\mathscr{T}[x(t)]\in\mathbb{R}/2\pi\mathbb{Z}$ is defined to be the angle from $\nu(a)$ to $\nu(b)$, where~$\nu(t)$ is a unit normal vector field along $x(t)$ which is parallel with respect to $\nabla^\perp$. The total torsion depends on the choice of $A$, $B$, but it is
conformally invariant once we fix the conformal classes of $A$ and $B$; see Proposition~\ref{conf-invariance-torsion}. If we use $A=s(\tilde x(a)), B=s(\tilde x(b))$, then we have
\[
\mathscr{L}[x(t)]\equiv 2\mathscr{T}[x(t)] \mod 2\pi\mathbb{Z}
\]
(Proposition~\ref{F-T}). Thus, conformal geodesics are also characterized as critical curves of the total torsion functional $\mathscr{T}$.

When the acceleration $\nabla_{\dot x}\dot x$ is nowhere proportional to the velocity $\dot x$, we can define the {\it torsion} $\tau(t)$ of $x(t)$, which is given by the formula
\[
\tau(t)=\frac{\det (\dot x, \nabla_{\dot x}\dot x, \nabla_{\dot x}\nabla_{\dot x}\dot x)}{|\dot x\times \nabla_{\dot x}\dot x|^2},
\]
and the total torsion is represented as the total integral of the torsion up to the addition of a~constant $c\in\mathbb{R}/2\pi\mathbb{Z}$ determined by $s$:
\[
\mathscr{T}[x(t)]\equiv-\int_a^b\tau(t)|\dot x(t)|\,{\rm d}t+c\mod 2\pi\mathbb{Z}.
\]
In summary, we have the following.

\begin{Theorem}\label{conf-geod-total-torsion}
A regular curve $x(t)$, $a\le t\le b$, on $\bigl(\Sigma^3, g\bigr)$ is a conformal geodesic if and only if it is a critical curve of the total torsion functional $\mathscr{T}$ under the variations fixing $(x(a), \dot x(a)/|\dot x(a)|)$ and $(x(b), \dot x(b)/|\dot x(b)|)$ with common $A$, $B$.

When $x(t)$ satisfies $\dot x\times \nabla_{\dot x}\dot x\neq0$, the functional $\mathscr{T} $ can be replaced by the integral functional
\[
\int_a^b \tau(t)|\dot x(t)|\,{\rm d}t=\int_a^b \frac{\det (\dot x, \nabla_{\dot x}\dot x, \nabla_{\dot x}\nabla_{\dot x}\dot x)}{|\dot x\times \nabla_{\dot x}\dot x|^2}|\dot x|\,{\rm d}t.
\]
\end{Theorem}
The variationality of conformal geodesics on $\mathbb{R}^3$ has been proved by
 Barros--Ferr\'andez~\cite{BF1, BF2}, and Theorem~\ref{conf-geod-total-torsion} generalizes their result.

The variation of the total torsion in the curved cases is computed by Chern--Kn\"oppel--Pedit--Pinkall~\cite{CKPP} in a different context, and Theorem~\ref{conf-geod-total-torsion} can also be proved directly from their formula; see Theorem~\ref{variation-T}.

After our work was completed, the author was informed of an independent work of
 Kruglikov--Matveev--Steneker~\cite{KMS}, which also proves Theorem~\ref{conf-geod-total-torsion}. Their proof is different from ours and is based on a detailed analysis of the Euler--Lagrange equation of the total torsion functional. Kruglikov~\cite{K} also proved that conformal geodesics are not variational in dimensions greater than three.

 Finally, we comment on the relation of our results to the general theory of Cartan geometry. Conformal geometry and CR geometry are both examples of so-called {\it parabolic geometry}, that~is, Cartan geometry modeled on a homogeneous space $G/P$ where $G$ is a semi-simple Lie group and $P$ is a parabolic subgroup. Following the pioneering works of Cartan and Tanaka, the general theory of parabolic geometry has been extensively developed; see~\cite{CS} for an exposition of the theory together with historical remarks.
In particular, Fefferman's construction in CR geometry has been generalized within the broader framework of extensions $G\subset \widetilde G$ of symmetry groups \cite{C, CS}. The twistor CR manifold can also be reinterpreted in this setting. In fact, LeBrun's twistor CR structure corresponds to the special case ($p=q=0$) of the twistor construction of CR structure from quaternionic contact structures associated with the extension $\mathrm{Sp}(p+1, q+1)\subset \mathrm{SU}(2p+2, 2q+2)$ described in \cite[Section~4.5.5]{CS}. We also remark that Sato--Yamaguchi~\cite{SY} introduced the twistor CR structure on the tangent sphere bundle over conformal 3-manifolds by a Cartan geometric method independently of LeBrun. Moreover, a general theory of distinguished families of curves in parabolic geometry, including CR chains and conformal geodesics, has been developed by \v{C}ap--Slov\'ak--\v Z\'adn\'ik~\cite{CSZ}. However, the correspondence between distinguished curves in the general setting of Fefferman-type constructions does not yet appear to be fully understood. Developing a general theory in this direction would therefore be of considerable interest, and the present paper may be viewed as a case study illustrating the correspondence between distinguished curves and their variational character.

This paper is organized as follows: In Section~\ref{CR}, we review the notion of CR structure and define (null) chains on CR manifolds via the Fefferman metric. In Section~\ref{conformal-three}, we recall the definition of conformal geodesic and introduce LeBrun's twistor CR manifold over conformal $3$-manifolds. Then, in Section~\ref{Fefferman-twistor-CR}, we give an explicit description of the Fefferman metric for twistor CR manifolds and compute its connection, curvature, and null geodesic equation to prove Theorem~\ref{chain-projection}\,(1). Here, we also prove that the canonical lift of a conformal geodesic becomes a~chain, which gives a partial proof to Theorem~\ref{chain-projection}\,(2); the full proof will be given in Section~\ref{sphere-bundle} after we reinterpret our constructions in the sphere/frame bundle picture. In~Section~\ref{variational-principle}, we give a variational characterization for conformal geodesics in dimension three via the Kropina metric for chains, and discuss its relation to the total torsion functional.

\begin{Notation}
We use tensor indices $\alpha, \beta =1, \dots, n$ for $T^{1, 0}M$, $i, j, k, l, m =1, 2, 3$ for $\Sigma$, and $a, b, c, d =0, 1, 2, \ol1, \ol2, 3$ for $\wh{\mathcal{C}}$. We adopt Einstein's summation convention throughout this paper.
For 1-forms $\varphi$, $\psi$, we define the symmetric product by
\[
(\varphi\cdot\psi)(V, W):=\frac{1}{2} (\varphi(V)\psi(W)+\varphi(W)\psi(V)).
\]
\end{Notation}

\section{CR manifolds and the Fefferman metric}\label{CR}
\subsection{CR structures}
An {\it almost CR structure} on a $(2n+1)$-dimensional $C^\infty$ manifold $M$ is defined to be a rank $n$ complex vector sub-bundle $T^{1, 0}M\subset\mathbb{C}TM$ satisfying $T^{1, 0}M\cap T^{0, 1}M=0$, where \smash{${T^{0, 1}M:=\ol{T^{1, 0}M}}$}. If the space of sections of $T^{1, 0}M$ is closed under the Lie bracket, it is called a {\it CR structure}. Since $T^{1, 0}M$ contains no non-zero real vector, the real part $H:=\operatorname{Re} T^{1, 0}M\subset TM$ becomes a~rank $2n$ real sub-bundle satisfying $\mathbb{C}H=T^{1, 0}M\oplus T^{0,1}M$.

We say a CR structure $T^{1, 0}M$ is {\it non-degenerate} if $H$ is a contact distribution, or equivalently, the {\it Levi form}
\[
T^{1, 0}_xM\times T^{0, 1}_xM\ni\bigl(Z, \ol W\bigr)\longmapsto {\rm i}\bigl[Z, \ol W\bigr]
\in T_xM/H_x
\]
gives a non-degenerate $TM/H$-valued hermitian form on \smash{$T_x^{1, 0}M$} at any point $x\in M$. We note that the right-hand side is well-defined via arbitrary extensions of $Z$, $\ol W$ to vector fields.

We usually assume that the real line bundle $TM/H$ is orientable so that we can trivialize $H^\perp \subset T^*M$ by a global contact form \smash{$\th\in \Gamma\bigl(H^\perp\bigr)$}. For a fixed orientation of $TM/H$, a contact form is determined up to multiplications by positive functions:
\smash{$\th\mapsto \wh\th={\rm e}^{\Upsilon}\th$, $\Upsilon\in C^\infty(M)$}. For each choice of $\th$, we have the real valued Levi form $h_\th$ by trivializing $TM/H$, which transforms conformally as $h_{\wh\th}={\rm e}^\Upsilon h_\th$.

The {\it Reeb vector field} of a contact form $\th$ is a real vector field $T$ uniquely
specified by the conditions $\th(T)=1$, $T\lrcorner\, {\rm d}\th=0$. Choosing a local frame $(Z_\a)$ for $T^{1, 0}M$, we have a local frame \smash{$\bigl(T, Z_\a, Z_{\ol\a}:=\ol{Z_\a}\bigr)$} for $\mathbb{C}TM=\mathbb{C}T\oplus T^{1, 0}M\oplus T^{0,1}M$ called an {\it admissible frame}, and its dual coframe \smash{$\bigl(\th, \th^\a, \th^{\ol\a}=\ol{\th^\a}\bigr)$}. In this coframe, we can write as
\[
{\rm d}\th={\rm i} h_{\a\ol\b}\th^\a\wedge\th^{\ol\b}
\]
and the Levi form is given by $\bigl(Z, \ol W\bigr)\mapsto h_{\a\ol\b}Z^\a \ol{W^\b}$.

\subsection{The Fefferman metric}\label{CR-Fefferman-metric}
The Fefferman metric associated to CR manifolds was first introduced by Fefferman~\cite{F} for the boundary of strictly pseudoconvex domains in $\mathbb{C}^{n+1}$ via an indefinite K\"ahler metric on $\mathbb{C}^{n+1}\times\mathbb{C}^*$ called the ambient metric. Burns--Diederich--Shnider~\cite{BDS} gave a construction via the CR Cartan bundle, and intrinsic characterizations were obtained by Farris~\cite{Fa} and Lee~\cite{Lee}.
Here, we~introduce the Fefferman metric by following Lee's formulation.

Let $\bigl(M, T^{1, 0}M\bigr)$ be a non-degenerate $(2n+1)$-dimensional CR manifold.
The {\it CR canonical bundle} is the complex line bundle over $M$ defined by
\[
K_M:=\biggl\{\xi\in\bigwedge\nolimits^{\!n+1}(\mathbb{C}T^*M)\, \Big|\, \ol Z\lrcorner\, \xi=0\ \textrm{for any}\ \ol Z\in T^{0,1}M\biggr\}.
\]
We note that a choice of an admissible frame gives a local frame $\th\wedge\th^1\wedge\cdots\wedge\th^n$ for $K_M$. The {\it Fefferman space} is the associated $S^1$-bundle
\[
\mathcal{C}:=K^\times_M/\mathbb{R}_+
\]
over $M$. If we choose a contact form $\th$, the Fefferman space is realized as a sub-bundle of $K^\times_M$~as%
\begin{equation}\label{embed-C}
\mathcal{C}\cong \bigl\{{\rm e}^{{\rm i}s}\th\wedge\th^1\wedge\cdots\wedge\th^n\mid s\in \mathbb{R}/2\pi\mathbb{Z}\bigr\} \subset K^\times_M,
\end{equation}
where $\bigl(\th, \th^\a, \th^{\ol\a}\bigr)$ is an admissible coframe such that $\bigl|\det\bigl(h_{\a\ol\b}\bigr)\bigr|=1$. This realization defines a~tautological $(n+1)$-form
\[
\xi={\rm e}^{{\rm i}s}\th\wedge\th^1\wedge\cdots\wedge\th^n
\]
on $\mathcal{C}$. Also, $\varphi={\rm e}^{{\rm i}s}\th^1\wedge\cdots\wedge\th^n$ becomes a global $n$-form on $\mathcal{C}$ uniquely characterized by the conditions
\[
\xi=\th\wedge\varphi, \qquad V\lrcorner\, \varphi=0\quad \text{for any lift}\ V\ \text{of}\ T;
\]
see {\cite[Lemma 3.3]{Lee}}.

We also have a $1$-form $\sigma$ on $\mathcal{C}$ specified by the following proposition.

\begin{Proposition}[{\cite[Proposition 3.4]{Lee}}]\label{sigma-condition}
There exists a unique real $1$-form $\sigma$ on $\mathcal{C}$ satisfying
\begin{align*}
&{\rm d}\xi={\rm i}(n+2)\sigma\wedge\xi, \\
&\sigma\wedge {\rm d}\varphi\wedge\ol\varphi=(\operatorname{Tr} {\rm d}\sigma){\rm i}\sigma\wedge\th\wedge\varphi\wedge\ol\varphi.
\end{align*}
\end{Proposition}
Here, for a 2-form $\omega$ on $M$ written as \smash{$\omega={\rm i}\omega_{\a\ol\b}\th^\a\wedge\th^{\ol\b}+\cdots$} in an admissible coframe, we define
\smash{$\operatorname{Tr} \omega:=h^{\a\ol\b}\omega_{\a\ol\b}$} using the inverse \smash{$h^{\a\ol\b}$} of the Levi form. The $1$-form $(n+2)\sigma$ gives a~principal connection form on the $S^1$-bundle $\mathcal{C}$.

By using $\sigma$, we now introduce the {\it Fefferman metric} as the pseudo-Riemannian metric on $\mathcal{C}$ defined by
\[
g^{\mathrm{F}}:=2h_{\a\ol\b}\theta^\a\cdot \theta^{\ol\b}+4\theta\cdot \sigma.
\]
We note that our definition of $g^{\mathrm{F}}$ is {\it twice} the definition of Lee. In the coframe $\smash{\bigl(\th^0:=\th, \th^\a, \th^{\ol\a}},\allowbreak \smash{\th^{n+1}:=2\sigma\bigr)}$, the non-zero components of the metric tensor are
\[
g^{\mathrm{F}}_{\a\ol\b}=g^{\mathrm{F}}_{\ol\b\a}=h_{\a\ol\b}, \qquad
g^{\mathrm{F}}_{0\, n+1}=g^{\mathrm{F}}_{n+1\, 0}=1.
\]
If the Levi form has signature $(p, q)$ as a hermitian form, then the Fefferman metric has signature $(2p+1, 2q+1)$. Moreover, if we rescale the contact form
as $\wh\th={\rm e}^\Upsilon\th$, the Fefferman metric transforms as
$\wh{g}^{\,\mathrm{F}}={\rm e}^\Upsilon g^{\mathrm{F}}$. Thus, the conformal structure $\bigl[g^{\mathrm{F}}\bigr]$ is independent of the choice of contact form~\cite[Theorem~5.17]{Lee}.

It follows from the formula in~\cite{Lee} of the connection form of $g^{\mathrm{F}}$ in terms of the pseudo-hermitian structure of $M$ that $g^{\mathrm{F}}$ is conformally flat if and only if $M$ is locally CR equivalent to the hyperquadric
\[
\mathcal{Q}=\bigl\{[\xi]\in \mathbb{CP}^{n+1} \,\big|\, -\bigl|\xi^0\bigr|^2-\bigl|\xi^1\bigr|^2-\cdots-
\bigl|\xi^q\bigr|^2+\bigl|\xi^{q+1}\bigr|^2+\cdots+\bigl|\xi^{n+1}\bigr|^2=0 \bigr\},
\]
which is the flat model of CR manifold.

\subsection{Chains and null chains}
It is well-known that null geodesics of a pseudo-Riemannian metric are conformally invariant up to reparametrizations. Hence, the projections of null geodesics of the Fefferman metric provide an invariant family of curves on CR manifolds. Let $\g(t)$ be a null geodesic of $g^{\mathrm{F}}$. Since $g^{\mathrm{F}}$ is $S^1$-invariant, the infinitesimal generator $K$ of the $S^1$-action on $\mathcal{C}$ is a Killing vector field and hence $g^{\mathrm{F}}(K, \dot\g(t))$ is constant in $t$.

\begin{Definition}[\cite{F, Koch}]
An unparametrized curve on $M$ is called a {\it chain} (resp.\ {\it null chain}) if it is the projection of a non-vertical null geodesic $\g(t)$ of $\bigl[g^{\mathrm{F}}\bigr]$ with $g^{\mathrm{F}}(K, \dot\g)\neq0$ \big(resp.\ $g^{\mathrm{F}}(K, \dot\g)=0$\big).
\end{Definition}

Since $K\lrcorner\, g^{\mathrm{F}}=2\th$, the definition implies that chains are transverse to the contact distribution~$H$ while null chains are tangent to $H$ and null with respect to the Levi form.

For each choice of a direction transverse to $H_x$, there exists a unique chain emanating in that direction. Thus, chains can be considered as `geodesics' in CR geometry. In fact, they are characterized as geodesics of a Kropina metric on CR manifolds, which is a Finsler metric singular along $H$; see~\cite{CMMM} and Section~\ref{Kropina} below. On the other hand, there exists a 1-parameter family of null chains which are tangent to a fixed null vector in $H_x$; see~\cite{Koch}.

\section{Conformal three-manifolds and twistor CR manifolds}\label{conformal-three}

\subsection{Conformal three-manifolds}
Let $(\Sigma, [g])$ be a $3$-dimensional conformal manifold (of Riemannian signature). We take a representative metric $g\in[g]$ and let $R_{ij}{}^k{}_l$ be the curvature tensor of $g$, defined by $(\nabla_i\nabla_j-\nabla_j\nabla_i)X^k=R_{ij}{}^k{}_lX^l$. Since the Weyl curvature always vanishes in three dimensions, we can write as
\begin{equation}\label{curvature-tensor}
R_{ijkl}=P_{ik}g_{jl}-P_{jk}g_{il}+P_{jl}g_{ik}-P_{il}g_{jk}
\end{equation}
with the Schouten tensor
\[
P_{ij}:=R_{ij}-\frac{1}{4}Rg_{ij}.
\]
Here $R_{ij}:=R_{ki}{}^k{}_j$ is the Ricci tensor and $R:=R_k{}^k$ is the scalar curvature.

The Cotton tensor $C=(C_{ijk})\in \Gamma\bigl(T^*\Sigma\otimes \wedge^2 T^*\Sigma\bigr)$ is defined by
\[
C(X, Y, Z):=(\nabla_Z P)(X, Y)-(\nabla_Y P)(X, Z), \qquad \text{i.e.,}\qquad
C_{ijk}=\nabla_k P_{ij}-\nabla_j P_{ik}.
\]
Since $\Sigma$ is 3-dimensional, the Cotton tensor is conformally invariant, and identically vanishes if and only if $(\Sigma, [g])$ is conformally flat.

If $\Sigma$ is oriented, we have the cross product $Z=X\times Y\in T_p\Sigma$ of tangent vectors $X, Y\in T_p\Sigma$. In the index notation, it is given by
\[
Z^i=\e^i{}_{jk}X^j Y^k,
\]
where $\e_{ijk}=\e_{[ijk]}$ is the volume form of $g$:
\[
vol_g=\frac{1}{3!}\e_{ijk}\th^i\wedge \th^j\wedge \th^k.
\]
Applying the Hodge star operator to the Cotton tensor, we define $*C\in\Gamma(T^*\Sigma\otimes T^*\Sigma)$ by
\[
(*C)_{ij}:=\frac{1}{2}\varepsilon_j{}^{kl}C_{ikl},
\]
which is characterized by the equation
\[
C(X, Y, Z)=(*C)(X, Y\times Z).
\]
Since $C_{ijk}$ is trace-free by the contracted second Bianchi identity, $(*C)_{ij}$ becomes a symmetric 2-tensor.

\subsection{Conformal geodesics}
Although geodesics of a Riemannian metric are not conformally invariant, there is an alternative conformally invariant family of curves on conformal manifolds, called {\it conformal geodesics}. They are characterized by the conformally invariant third order ODE
\begin{equation}\label{conf-geod-proj}
\nabla_{\dot x}\nabla_{\dot x}\dot x-\frac{3g(\dot x, \nabla_{\dot x}\dot x)}{|\dot x|^2}\nabla_{\dot x}\dot x+\frac{3|\nabla_{\dot x}\dot x|^2}{2|\dot x|^2}\dot x
+2P(\dot x, \dot x)\dot x-|\dot x|^2P(\dot x)=0.
\end{equation}
Here, we set $P(\dot x, \dot x):=P_{ij}\dot x^i\dot x^j$ and \smash{$P(\dot x)^i:=P^i{}_j \dot x^j$}. We refer the reader to~\cite{BEG} for a derivation of this equation by using the tractor calculus, which is an invariant calculus in conformal geometry.
Conformal geodesics are also called {\it conformal circles} since these curves are circles and straight lines when $g$ is the Euclidean metric.

It is shown in~\cite{BE} that if $x(t)$ is a solution to the equation \eqref{conf-geod-proj}, a reparametrized curve $\tilde x(t)=x(\varphi(t))$ solves \eqref{conf-geod-proj} if and only if $\varphi(t)$ is a projective linear transformation: $\varphi(t)=(at+b)/(ct+d)$. Hence, we call a solution to \eqref{conf-geod-proj}
a {\it conformal geodesic with a projective parameter}. In~\cite{BE}, they also prove that one can define an equation for `unparametrized conformal geodesics' by taking the $\dot x$-orthogonal part of \eqref{conf-geod-proj}.
\begin{Proposition}[{\cite[Proposition 4.2]{BE}}]
A regular curve $x(t)$ can be reparametrized to satisfy~\eqref{conf-geod-proj} if and only if it satisfies the equation
\begin{equation}\label{conf-geod}
\dot x\wedge\biggl(\nabla_{\dot x}\nabla_{\dot x}\dot x-\frac{3g(\dot x, \nabla_{\dot x}\dot x)}{|\dot x|^2}\nabla_{\dot x}\dot x-|\dot x|^2P(\dot x)\biggr)=0.
\end{equation}
\end{Proposition}
In this paper, we always consider conformal geodesics as unparametrized curves, and we say a~regular curve $x(t)$ is a {\it conformal geodesic} if it satisfies \eqref{conf-geod}.
We note that if $x(t)$ has a constant speed $|\dot x|$, the equation \eqref{conf-geod} is reduced to
\begin{equation}\label{conf-geod-const-speed}
\dot x\wedge\bigl(\nabla_{\dot x}\nabla_{\dot x}\dot x-|\dot x|^2P(\dot x)\bigr)=0
\end{equation}
since $g(\dot x, \nabla_{\dot x}\dot x)=0$.

Conformal geodesics are also characterized as follows.

\begin{Proposition}[{\cite[Proposition 3.3]{BE}}]\label{conf-geod-characterization}
A regular curve $x(t)$ is a conformal geodesic if and only if there exist a metric $g\in[g]$ and a reparametrization such that
\[
\nabla_{\dot x}\dot x=P(\dot x)=0.
\]
\end{Proposition}

\subsection{Twistor CR manifolds}
We will define the twistor CR manifold $M$ over $\bigl(\Sigma^3, [g]\bigr)$, which was introduced by LeBrun~\cite{LeB} as a $3$-dimensional analogue of Penrose's twistor space for anti-self dual conformal $4$-manifolds.

We extend a representative metric $g\in[g]$ on $T^*\Sigma$ to a complex bi-linear form on $\mathbb{C}T^*\Sigma$ and define a $7$-dimensional manifold $\hat M$ as the set of complex null covectors:
\[
\hat M:=\{\zeta\in\mathbb{C}T^*\Sigma\mid g(\zeta, \zeta)=0,\, \zeta\neq0\}.
\]
By taking the quotient by the $\mathbb{C}^*$-action, we obtain a $5$-dimensional manifold
\[
M:=\hat M/\mathbb{C}^*\subset \mathbb{P}(\mathbb{C}T^*\Sigma).
\]
We denote the natural projections by $\tilde\pi\colon \hat M\to M$, $\pi\colon M\to \Sigma$. Note that the fiber of $\pi$ is biholomorphic to \smash{$\bigl\{[\zeta]\in\mathbb{CP}^2\mid \zeta_1^2+\zeta_2^2+\zeta_3^2=0\bigr\}\cong \mathbb{CP}^1$}.

Let
\smash{$\alpha:=\bigl(\zeta_i {\rm d}x^i\bigr)|_{T\hat M}$} be the canonical (tautological) $1$-form on
$\hat M$ and set $\omega:={\rm d}\alpha$. Then, the distribution
\[
 D:=\textrm{Ker}\, \omega=\bigl\{v\in \mathbb{C}T\hat M\mid v\lrcorner\, \omega=0\bigr\}
\]
is $\mathbb{C}^*$-invariant and we can define the distribution $\tilde\pi_* D\subset \mathbb{C}TM$. LeBrun~\cite{LeB} proved that
\[
T^{1, 0}M:=\ol{\tilde\pi_* D}
\]
becomes a non-degenerate CR structure on $M$ whose Levi form has signature $(1, 1)$. The construction depends only on the conformal class $[g]$, and $\bigl(M, T^{1, 0}M\bigr)$ is called the {\it twistor CR manifold} over $(\Sigma, [g])$.

The CR structure $T^{1, 0}M$ can be described in terms of coordinates as follows.
Fix a point $x_0\in \Sigma$ and let $(x^i)$ be normal coordinates centered at $x_0$ with respect to a representative metric $g\in[g]$.
\begin{Proposition}\label{D}
For any point $[\zeta]\in \pi^{-1}(x_0)$, we have
\[
\ol D_{\zeta}=\mathbb{C}\ol{\zeta^i}\frac{\partial}{\partial x^i}\oplus \biggl\{ w_i\frac{\partial}{\partial \zeta_i} \, \bigg|\, \zeta_i w^i=0\biggr\},
\]
and the projection \smash{$\tilde\pi_*\colon \ol D_{\zeta}\to T^{1. 0}_{[\zeta]} M$} induces the isomorphism
\[
T^{1, 0}_{[\zeta]} M\cong\mathbb{C}\ol{\zeta^i}\frac{\partial}{\partial x^i}\oplus \biggl(\biggl\{ w_i\frac{\partial}{\partial \zeta_i}\, \bigg|\, \zeta_i w^i=0\biggr\}\Big/\mathbb{C}\zeta_i\frac{\partial}{\partial \zeta_i}\biggr).
\]
\end{Proposition}
See \cite[Proposition 6.1]{Mar} for the proof.

\section[The Fefferman metric and (null) chains of twistor CR manifolds]{The Fefferman metric and (null) chains\\ of twistor CR manifolds}\label{Fefferman-twistor-CR}
Let $\pi\colon M\to \Sigma$ be the twistor CR manifold over a conformal $3$-manifold $(\Sigma, [g])$. We will describe the Fefferman metric for $M$ and compute its geometric quantities.

We fix a representative metric $g\in[g]$ and define
\[
\wh{\mathcal{C}}:=\bigl\{\zeta\in \hat M\mid g\bigl(\zeta, \ol\zeta\bigr)=1\bigr\}
=\bigl\{\zeta\in\mathbb{C}T^*\Sigma\mid g(\zeta, \zeta)=0,\, g\bigl(\zeta, \ol\zeta\bigr)=1\bigr\},
\]
which is an $S^1$-bundle over $M$: \smash{$\wh{\mathcal{C}}/\mathrm{U}(1)=M$}.
We show that \smash{$\wh{\mathcal{C}}$} gives a double covering of the Fefferman space $\mathcal{C}$ and construct the Fefferman metric on \smash{$\wh{\mathcal{C}}$} instead of $\mathcal{C}$.

\subsection[An adapted frame on wh{C}]{An adapted frame on $\boldsymbol{\wh{\mathcal{C}}}$}
Let $(\zeta_i)$ be the fiber coordinates for $\zeta_i{\rm d}x^i\in \mathbb{C}T^*\Sigma$ associated to local coordinates $(x^i)$ on $\Sigma$. On~$\wh{\mathcal{C}}$, it holds that
\[
\zeta^i\zeta_i=0, \qquad \zeta^i\ol\zeta_i=1.
\]
By using the metric, we identify $\zeta_i {\rm d}x^i$ with the complex tangent vector
\[
\zeta=\zeta^i\frac{\partial}{\partial x^i}
\]
and regard that $\zeta$ also represents the global section of the vector bundle
$\hat\pi^{-1}(\mathbb{C}T\Sigma)$ over $\wh{\mathcal{C}}$, where \smash{$\hat\pi\colon \wh{\mathcal{C}}\to \Sigma$} is the projection. We (locally) choose an orientation of $\Sigma$ and define a real vector
\[
\eta=\eta^i\frac{\partial}{\partial x^i}\in \hat\pi^{-1}(T\Sigma)
\]
by
\[
\eta:={\rm i}\zeta\times \ol\zeta=2\operatorname{Re}\zeta\times \operatorname{Im}\zeta.
\]
Using the volume form $\e_{ijk}$ of $g$, we can write as
\[
\eta^j={\rm i}\e^{jkl}\zeta_k\ol{\zeta_l}.
\]
The property of the cross product gives
\begin{alignat*}{3}
&g(\eta, \eta)=1, \qquad&& g(\eta, \zeta)=g\bigl(\eta, \ol\zeta\bigr)=0,& \\
&\eta\times\ol\zeta={\rm i}\ol\zeta, \qquad&& \zeta\times\eta={\rm i}\zeta.&
\end{alignat*}
(In the index notation, one can use the identity $\e_{ijk}\e^i{}_{lm}=g_{jl}g_{km}-g_{kl}g_{jm}$ to verify these equations.)

Since the conditions $g(\zeta, \zeta)=0$, $g\bigl(\zeta, \ol\zeta\bigr)=1$ are preserved by the parallel transport of $\zeta$, we can define the horizontal lifts of tangent vectors on $\Sigma$ to \smash{$\wh{\mathcal{C}}$}. Let \smash{$\Gamma_{ij}{}^k$} be the Christoffel symbols of the Levi-Civita connection of $g$. Then, the horizontal lift of the tangent vector $\partial/\partial x^i$ is given~by
\[
\frac{\d}{\d x^i}:=\frac{\partial}{\partial x^i}+\Gamma_{ij}{}^k\zeta_k\frac{\partial}{\partial \zeta_j}+\Gamma_{ij}{}^k\ol{\zeta_k}\frac{\partial}{\partial \ol{\zeta}_j}\in T\wh{\mathcal{C}}.
\]
Note that this vector transforms tensorially in the index $i$ under coordinate changes. The annihilators of the horizontal distribution on $\wh{\mathcal{C}}$ are given by
\[
\d \zeta_i:={\rm d}\zeta_i-\Gamma_{ki}{}^j \zeta_j {\rm d}x^k, \qquad \d \ol{\zeta}_i:={\rm d}\ol{\zeta}_i-\Gamma_{ki}{}^j \ol{\zeta}_j {\rm d}x^k,
\]
which also transform tensorially.

Using these notations, we define a frame
\smash{$\bigl(\tilde T, \tilde Z_1, \tilde Z_2, \tilde Z_{\ol1}, \tilde Z_{\ol2}, K\bigr)$}
 for $\mathbb{C}T\wh{\mathcal{C}}$ and its dual coframe \smash{$\bigl(\th, \tilde\th^1, \tilde\th^2, \tilde\th^{\ol1}, \tilde\th^{\ol2}, \tilde\th^3\bigr)$} by
\[
\tilde T:=\eta^i \frac{\d}{\d x^i}, \qquad \tilde Z_1:=\ol{\zeta^i}\frac{\d}{\d x^i}, \qquad \tilde Z_2:=-{\rm i}\eta_j\frac{\partial}{\partial \zeta_j}, \qquad
K:={\rm i}\biggl(\zeta_j\frac{\partial}{\partial \zeta_j}
-\ol{\zeta}_j\frac{\partial}{\partial \ol{\zeta}_j}\biggr)
\]
and
\[
\th:=\eta_i {\rm d}x^i, \qquad \tilde\th^1:=\zeta_i {\rm d}x^i, \qquad \tilde\th^2:={\rm i}\eta^j\d\zeta_j, \qquad \tilde\th^3:=-\frac{{\rm i}}{2}\bigl(\ol{\zeta^j}\d\zeta_j-\zeta^j\d\ol{\zeta}_j\bigr)=-i\ol{\zeta^j}\d\zeta_j.
\]
Note that these vectors and forms are independent of the choice of coordinates and globally defined if $\Sigma$ is globally oriented. We also note that $K$ agrees with the infinitesimal generator of the $S^1$-action on $\wh{\mathcal{C}}$.

We shall prove that this indeed gives a (co)frame which is `adapted' to the CR structure on~$M$.

\begin{Proposition}\quad
\begin{itemize}\itemsep=0pt
\item[$(1)$] The $1$-from $\th$ descends to a contact form on $M$.
\item[$(2)$] The $1$-forms $\tilde\th^1$, $\tilde\th^2$ are horizontal with respect to $\wh{\mathcal{C}}\to M$, and for any local section ${s\colon M\to \wh{\mathcal{C}}}$, the set of $1$-forms $\bigl(\th, \th^1=s^*\tilde\th^1, \th^2=s^*\tilde\th^2, \th^{\ol1}=\ol{\th^1}, \th^{\ol2}=\ol{\th^2}\bigr)$ gives an admissible coframe on $M$.
\end{itemize}
\end{Proposition}
\begin{proof}
We take normal coordinates $(x^i)$ around a fixed point $x_0\in \Sigma$ and work at a point $p\in\pi^{-1}(x_0)\subset M$.

(1) Since each coefficient $\eta_i$ is a $\mathrm{U}(1)$-invariant function, $\th$ descends to a $1$-form on $M$. Moreover, by Proposition~\ref{D}, $\th$ annihilates $T^{1, 0}M\oplus T^{0, 1}M$ hence is a contact form on $M$. We note that the projection $T$ of $\tilde T$ to $M$ is well-defined since $\tilde T$ is $\mathrm{U}(1)$-invariant, and it gives the Reeb vector field of $\th$. In fact, $T$ is represented as $T=\eta^i\bigl(\partial/\partial x^i\bigr)$ at $p$, and since
${\rm d}\e_{ijk}(x_0)=0$, ${\rm d}g_{ij}(x_0)=0$, we have ${\rm d}\eta_i(T)=0$, $\eta^i{\rm d}\eta_i=0$ at $p$ and hence $\th(T)=\eta_i\eta^i=1$, $T\lrcorner\, {\rm d}\th=T\lrcorner\, \bigl({\rm d}\eta_i\wedge {\rm d}x^i\bigr)=0$.

(2) Since $\tilde\th^1(K)=\tilde\th^2(K)=0$, these forms are horizontal. By Proposition~\ref{D}, we see that~$\th^1$ and~$\th^2$ annihilate $T^{0, 1}M$. Moreover, $\tilde\th^1\bigl(\tilde T\bigr)=\tilde\th^2\bigl(\tilde T\bigr)=0$ implies \smash{$\th^1(T)=\th^2(T)=0$}. Thus, \smash{$\bigl(\th, \th^1, \th^2, \th^{\ol1}, \th^{\ol2}\bigr)$} is an admissible coframe.
\end{proof}

We remark that $\tilde\th^1$ and $\tilde\th^2$ do not descend to $1$-forms on $M$. In fact, if we denote the action of $\lambda\in \mathrm{U}(1)$ on $\wh{\mathcal{C}}$ by $\d_\lambda\colon \wh{\mathcal{C}}\to \wh{\mathcal{C}}$, we have
\[
\d_\lambda^* \tilde\th^1=\lambda \tilde\th^1, \qquad \d_\lambda^*\tilde\th^2=\lambda\tilde\th^2.
\]
Thus, $\bigl(\tilde\th^1, \tilde\th^2\bigr)$ can be considered as a `twisted' or `weighted' coframe for $T^{1, 0}M$.

We also note that if we replace $g$ with $\wh g={\rm e}^{2\Upsilon} g$, $\Upsilon\in C^\infty(\Sigma)$, we have $\wh\eta^{\,i}={\rm e}^{-\Upsilon}\eta^i$, $\wh\eta_i={\rm e}^{\Upsilon}\eta_i$ and hence the associated contact form is rescaled as $\wh\th={\rm e}^{\Upsilon}\th$.


\subsection{The differentials of the adapted coframe}
We will compute the exterior derivatives of the adapted coframe defined above.
We use the following lemma in computations in normal coordinates.
\begin{Lemma}
Let $(x^i)$ be normal coordinates of $g$ centered at a point $x_0\in\Sigma$. Then, we have
\begin{equation}\label{d-eta}
{\rm d}x^j=\eta^j\theta+\ol{\zeta^j}\tilde\theta^1+\zeta^j\tilde\theta^{\ol1}, \qquad
{\rm d}\zeta_j=-{\rm i}\bigl(\eta_j\tilde\theta^2-\zeta_j\tilde\theta^3\bigr), \qquad
{\rm d}\eta^j={\rm i}\bigl(\ol{\zeta^j}\tilde\theta^2-\zeta^j\tilde\theta^{\ol2}\bigr)
\end{equation}
at any point $p\in \pi^{-1}(x_0)\subset M$.
\end{Lemma}
\begin{proof}
At $p$, we have
\[
{\rm d}x^j\bigl(\tilde T\bigr)=\eta^j, \qquad {\rm d}x^j\bigl(\tilde Z_1\bigr)=\ol{\zeta^j},\qquad {\rm d}x^j\bigl(\tilde Z_2\bigr)={\rm d}x^j(K)=0,
\]
from which we obtain the first equation in \eqref{d-eta}. Similarly, we have
\[
{\rm d}\zeta_j\bigl(\tilde T\bigr)={\rm d}\zeta_j\bigl(\tilde Z_{1}\bigr)={\rm d}\zeta_j\bigl(\tilde Z_{\ol1}\bigr)={\rm d}\zeta_j\bigl(\tilde Z_{\ol2}\bigr)=0, \qquad
{\rm d}\zeta_j\bigl(\tilde Z_2\bigr)=-{\rm i}\eta_j, \qquad
{\rm d}\zeta_j(K)={\rm i}\zeta_j,
\]
and obtain the formula for ${\rm d}\zeta_j$. Finally, using $d\e_{ijk}(x_0)=0$, we compute as
 \[
{\rm d}\eta={\rm i}\,{\rm d}\zeta\times \ol\zeta+{\rm i}\zeta\times {\rm d}\ol\zeta
=\bigl(\eta\tilde\theta^2-\zeta\tilde\theta^3\bigr)\times\ol\zeta-\zeta\times\bigl(\eta\tilde\theta^{\ol2}-\ol\zeta\tilde\theta^3\bigr)={\rm i}\ol\zeta\tilde\theta^2-{\rm i}\zeta\tilde\theta^{\ol2}
\]
at $p$.
\end{proof}

\begin{Proposition} We have
\begin{gather}
{\rm d}\theta={\rm i}\bigl(\tilde\theta^1\wedge\tilde\theta^{\ol2}+\tilde\theta^2\wedge\tilde\theta^{\ol1}\bigr), \nonumber\\
{\rm d}\tilde\theta^1={\rm i}\theta\wedge\tilde\theta^2-{\rm i}\tilde\theta^1\wedge\tilde\theta^3, \nonumber\\
{\rm d}\tilde\theta^2={\rm i}\bigl(P\bigl(\zeta, \ol\zeta\bigr)+P(\eta, \eta)\bigr)\theta\wedge\tilde\theta^1+{\rm i}P(\zeta, \zeta)\theta\wedge\tilde\theta^{\ol1}
- {\rm i}P(\zeta, \eta)\tilde\theta^1\wedge\tilde\theta^{\ol1}-{\rm i}\tilde\theta^2\wedge\tilde\theta^3, \nonumber\\
{\rm d}\tilde\theta^3=- {\rm i}P\bigl(\ol\zeta, \eta\bigr)\theta\wedge\tilde\theta^1+{\rm i}
P(\zeta, \eta)\theta\wedge\tilde\theta^{\ol1}
+ 2{\rm i}P\bigl(\zeta, \ol\zeta\bigr)\tilde\theta^1\wedge\tilde\theta^{\ol1}+{\rm i}\tilde\theta^2\wedge\tilde\theta^{\ol2}.\label{diff-coframe}
\end{gather}
\end{Proposition}
\begin{proof}
We compute in normal coordinates around $x_0\in\Sigma$. By \eqref{d-eta}, we have
\begin{align*}
&{\rm d}\theta={\rm d}\eta_i\wedge {\rm d}x^i={\rm i}\bigl(\tilde\theta^1\wedge\tilde\theta^{\ol2}+\tilde\theta^2\wedge\tilde\theta^{\ol1}\bigr), \\
&{\rm d}\tilde\th^1={\rm d}\zeta_i\wedge {\rm d}x^i={\rm i}\theta\wedge\tilde\theta^2-{\rm i}\tilde\theta^1\wedge\tilde\theta^3.
\end{align*}
To compute ${\rm d}\tilde\theta^2$, we note that
\begin{align*}
{\rm d}\bigl(\Gamma_{kjl} {\rm d}x^k\bigr)&=\frac{1}{2}R_{kmlj}{\rm d}x^k\wedge {\rm d}x^m \\
&=\frac{1}{2}(P_{kl}g_{mj}-P_{ml}g_{kj}+P_{mj}g_{kl}-P_{kj}g_{ml}){\rm d}x^k\wedge {\rm d}x^m \\
&=(P_{kl}g_{mj}+P_{mj}g_{kl}){\rm d}x^k\wedge {\rm d}x^m
\end{align*}
holds at $p\in \pi^{-1}(x_0)$ since $\Gamma_{ij}{}^k(x_0)=0$. Using this equation and \eqref{d-eta}, we have
\begin{align*}
{\rm d}\tilde\theta^2={}&{\rm i}\,{\rm d}\eta^j\wedge {\rm d}\zeta_j-{\rm i}\eta^j\zeta^l{\rm d}\bigl(\Gamma_{kjl} {\rm d}x^k\bigr) \\
={}&{\rm i}\,{\rm d}\eta^j\wedge {\rm d}\zeta_j-{\rm i}\eta^j\zeta^l (P_{kl}g_{mj}+P_{mj}g_{kl}){\rm d}x^k\wedge {\rm d}x^m \\
={}&{\rm i}\bigl(\ol{\zeta^j}\tilde\th^2-\zeta^j\tilde\th^{\ol2}\bigr)\wedge\bigl(\eta_j\tilde\theta^2-\zeta_j\tilde\theta^3\bigr) \\
&{-}\,{\rm i}\eta^j\zeta^l (P_{kl}g_{mj}+P_{mj}g_{kl})\bigl(\eta^k\theta+\ol{\zeta^k}\tilde\theta^1+\zeta^k\tilde\theta^{\ol1}\bigr)\wedge \bigl(\eta^m\theta+\ol{\zeta^m}\tilde\theta^1+\zeta^m\tilde\theta^{\ol1}\bigr) \\
={}&{\rm i}\bigl(P\bigl(\zeta, \ol\zeta\bigr)+P(\eta, \eta)\bigr)\theta\wedge\tilde\theta^1+{\rm i}P(\zeta, \zeta)\theta\wedge\tilde\theta^{\ol1} -{\rm i}P(\zeta, \eta)\tilde\theta^1\wedge\tilde\theta^{\ol1}-{\rm i}\tilde\theta^2\wedge\tilde\theta^3.
\end{align*}
The computation of ${\rm d}\tilde\th^3$ is similar, and we omit it.
\end{proof}

It follows from the first equation in \eqref{diff-coframe} that
\[
{\rm d}\th={\rm i}\bigl(\theta^1\wedge\theta^{\ol2}+\theta^2\wedge\theta^{\ol1}\bigr)
\]
holds for the admissible coframe $\th^1=s^*\tilde \th^1$, $\th^2=s^*\tilde \th^2$ determined by a local section $s\colon M\to \wh{\mathcal{C}}$. Hence the Levi form is given by
\[
\bigl(h_{\a\ol\b}\bigr)=\begin{pmatrix} 0 & 1 \\ 1 & 0 \end{pmatrix},
\]
which indeed has signature $(1, 1)$.

\subsection{The Fefferman space for twistor CR manifolds}
Let $\mathcal{C}=K_M^\times/\mathbb{R}_+$ be the Fefferman space over $M$. For simplicity, we assume that $\Sigma$ is globally oriented, and let $\th=\eta_i{\rm d}x^i$ be the contact form determined by a choice of representative metric $g\in[g]$ as above.
We embed $\mathcal{C}$ into $K^\times_M$ as in \eqref{embed-C} by $\th$, and define the mapping
\[
\varpi\colon \ \wh{\mathcal{C}}\longrightarrow \mathcal{C}, \qquad u\longmapsto \bigl(\theta\wedge\tilde\theta^1\wedge\tilde\theta^2\bigr)_u
\]
by using our twisted coframe $\bigl(\tilde\th^1, \tilde\th^2\bigr)$ for $T^{1, 0}M$. 
Since we have
\[
\d_\lambda^*\bigl(\theta\wedge\tilde\theta^1\wedge\tilde\theta^2\bigr)
=\lambda^2\bigl(\theta\wedge\tilde\theta^1\wedge\tilde\theta^2\bigr)
\]
for the action of $\lambda\in\mathbb{C}^*$, the mapping $\varpi$ becomes a double covering map.

We will compute the pullback $\varpi^*g^{\mathrm{F}}$ of the Fefferman metric, which we also denote by $g^{\mathrm{F}}$ and call the {\it Fefferman metric} on $\wh{\mathcal{C}}$. We set
\[
\hat{\xi}:=\varpi^*\xi=\theta\wedge\tilde\theta^1\wedge\tilde\theta^2,\qquad \hat{\varphi}:=\varpi^*\varphi, \qquad \hat{\sigma}:=\varpi^*\sigma,
\]
where $\xi$, $\varphi$, $\sigma$ are the differential forms defined in Section~\ref{CR-Fefferman-metric}.
\begin{Lemma}
We have $\hat\varphi=\tilde\theta^1\wedge\tilde\theta^2$ and $\hat\sigma=\frac{1}{2}\tilde\th^3$.
\end{Lemma}
\begin{proof}
The 2-form $\hat\varphi=\tilde\theta^1\wedge\tilde\theta^2$ satisfies
$\hat\xi=\th\wedge\hat\varphi$ and $V\lrcorner\,\hat\varphi=0$ for any lift $V$ of $T$ since $\tilde T\lrcorner\,\hat\varphi= K\lrcorner\,\hat\varphi=0$. Thus, it satisfies the required characterization. By Proposition~\ref{sigma-condition}, $\hat\sigma$ is characterized by the conditions
\begin{align*}
&{\rm d}\hat\xi=4{\rm i}\hat\sigma\wedge\hat\xi, \\
&\hat\sigma\wedge {\rm d}\hat\varphi\wedge\ol{\hat\varphi}=(\operatorname{Tr} {\rm d}\hat\sigma){\rm i}\hat\sigma\wedge\th\wedge\hat\varphi\wedge\ol{\hat\varphi}.
\end{align*}
By \eqref{diff-coframe}, we have
\[
{\rm d}\hat\xi=2{\rm i}\tilde\theta^3\wedge\hat\xi, \qquad
\tilde\theta^3\wedge {\rm d}\hat\varphi\wedge \ol{\hat\varphi}=0.
\]
Moreover, since
\begin{align*}
{\rm d}\tilde\theta^3 =2{\rm i}P\bigl(\zeta, \ol\zeta\bigr)\tilde\theta^1\wedge\tilde\theta^{\ol1}+{\rm i}\tilde\theta^2\wedge\tilde\theta^{\ol2}+\cdots
 =2{\rm i}P\bigl(\zeta, \ol\zeta\bigr)\theta^1\wedge\theta^{\ol1}+{\rm i}\theta^2\wedge\theta^{\ol2}+\cdots,
\end{align*}
we have $\operatorname{Tr}{\rm d}\tilde\theta^3=0$. Thus, we obtain \smash{$\hat\sigma=\frac{1}{2}\tilde\theta^3$}.
\end{proof}

Since \smash{$g^{\mathrm{F}}=2h_{\a\ol\b}\theta^\a\cdot \theta^{\ol\b}+4\theta\cdot \hat\sigma=2h_{\a\ol\b}\tilde\theta^\a\cdot \tilde\theta^{\ol\b}+4\theta\cdot \hat\sigma$}, we obtain the following theorem.

\begin{Theorem}\label{Fefferman-metric}
The Fefferman metric on the double covering $\wh{\mathcal{C}}$ 
is given by
\[
g^{\mathrm{F}}=2\bigl(\tilde\theta^1\cdot\tilde\theta^{\ol2}+\tilde\theta^2\cdot\tilde\theta^{\ol1}+\theta\cdot \tilde\theta^3\bigr)
\]
in the adapted coframe.
\end{Theorem}

The non-zero components of $g^{\mathrm{F}}$ in the adapted coframe \smash{$\bigl(\tilde\th^0=\th, \tilde\th^1, \tilde\th^{2}, \tilde\th^{\ol1}, \tilde\th^{\ol2}, \tilde\th^3\bigr)$} are
\begin{equation}\label{metric-tensor}
g^{\mathrm{F}}_{1\ol2}=g^{\mathrm{F}}_{\ol2 1}=g^{\mathrm{F}}_{2\ol1}=g^{\mathrm{F}}_{\ol1 2}=g^{\mathrm{F}}_{03}=g^{\mathrm{F}}_{30}=1.
\end{equation}

\subsection[The connection form of g\^{}F]{The connection form of $\boldsymbol{g^{\mathrm{F}}}$}
The connection 1-forms $\omega_a{}^b$ of the Fefferman metric are computed as follows.

\begin{Proposition}\label{connection}
In the coframe \smash{$\bigl(\tilde\theta^0=\theta, \tilde\theta^1, \tilde\theta^{2}, \tilde\theta^{\ol1}, \tilde\theta^{\ol2}, \tilde\theta^3\bigr)$}, the connection forms of the Levi-Civita connection of $g^{\rm F}$ are given by
\begin{gather*}
\begin{alignedat}{7}
&\omega_0{}^0=0,\qquad&& \omega_1{}^0=\frac{{\rm i}}{2}\tilde\theta^{\ol2},\qquad&& \omega_2{}^0=\frac{{\rm i}}{2}\tilde\theta^{\ol1},\qquad&& \omega_3{}^0=0,& \\
&\omega_0{}^1=\frac{{\rm i}}{2}\tilde\theta^2,\qquad&& \omega_1{}^1=-\frac{{\rm i}}{2}\tilde\theta^3,\qquad&& \omega_2{}^1=-\frac{{\rm i}}{2}\theta,\qquad&& \omega_{\ol1}{}^1=0,\qquad && \omega_{\ol2}{}^1=0,\qquad&& \omega_3{}^1=\frac{{\rm i}}{2}\tilde\theta^1,&
\end{alignedat}
\\
\omega_0{}^2={\rm i}P(\zeta, \eta)\theta+{\rm i}P\bigl(\zeta, \ol\zeta\bigr)\tilde\theta^1+{\rm i}P(\zeta, \zeta)\tilde\theta^{\ol1},\\
\omega_1{}^2=-{\rm i}P(\eta, \eta)\theta-{\rm i}P\bigl(\ol\zeta, \eta\bigr)\tilde\theta^1-{\rm i}P(\zeta, \eta)\tilde\theta^{\ol1}, \\
\omega_2{}^2=-\frac{{\rm i}}{2}\tilde\theta^3, \qquad
\omega_{\ol1}{}^2=0, \qquad \omega_{\ol2}{}^2=0, \qquad \omega_3{}^2=\frac{{\rm i}}{2}\tilde\theta^2, \\
\omega_0{}^3=0, \qquad \omega_1{}^3={\rm i}P\bigl(\ol\zeta, \eta\bigr)\theta+{\rm i}P\bigl(\ol\zeta, \ol\zeta\bigr)\tilde\theta^1+{\rm i}P\bigl(\zeta, \ol\zeta\bigr)\tilde\theta^{\ol1}, \qquad \omega_2{}^3=\frac{{\rm i}}{2}\tilde\theta^{\ol2}, \qquad
\omega_3{}^3=0
\end{gather*}
and the complex conjugates of these forms.
\end{Proposition}
\begin{proof}
It suffices to check that these forms satisfy the structure equations
\[
g^{\mathrm{F}}_{bc}\omega_a{}^c+g^{\mathrm{F}}_{ac}\omega_b{}^c=dg^{\mathrm{F}}_{ab}, \qquad
{\rm d}\tilde\theta^a=\tilde\theta^b\wedge\omega_b{}^a.
\]
Since the metric tensor is given by \eqref{metric-tensor}, the first equation is equivalent to
\begin{alignat}{5}
&\omega_0{}^3=0,\qquad&& \omega_0{}^{\ol2}+\omega_1{}^3=0,\qquad&& \omega_0{}^{\ol1}+\omega_2{}^3=0,\qquad&& \omega_0{}^0+\omega_3{}^3=0, &\nonumber\\
&\omega_1{}^{\ol2}=0,\qquad&& \omega_1{}^{\ol1}+\omega_2{}^{\ol2}=0,\qquad&& \omega_1{}^{2}+\omega_{\ol1}{}^{\ol2}=0,\qquad&& \omega_1{}^1+\omega_{\ol2}{}^{\ol2}=0,& \nonumber\\
&\omega_1{}^0+\omega_3{}^{\ol2}=0,\qquad&& \omega_2{}^{\ol1}=0,\qquad&& \omega_2{}^{1}+\omega_{\ol2}{}^{\ol1}=0,\qquad&& \omega_2{}^0+\omega_3{}^{\ol1}=0,& \nonumber\\
&\omega_3{}^0=0&\label{metric-connection}
\end{alignat}
and these are satisfied. We can also verify the equation
\[
{\rm d}\tilde\theta^a=\theta\wedge\omega_0{}^a+\tilde\theta^1\wedge\omega_1{}^a+\tilde\theta^2\wedge\omega_2{}^a+\tilde\theta^{\ol1}\wedge\omega_{\ol1}{}^a+\tilde\theta^{\ol2}\wedge\omega_{\ol2}{}^a+\tilde\theta^3\wedge\omega_3{}^a
\]
for $a=0, 1, 2, 3$ by using \eqref{diff-coframe}.
\end{proof}

\subsection[The curvature form of g\^{}F]{The curvature form of $\boldsymbol{g^{\mathrm{F}}}$}
It is straightforward to compute the curvature form
\[
\Omega_a{}^b={\rm d}\omega_a{}^b-\omega_a{}^c\wedge\omega_c{}^b
=\frac{1}{2}R^{\mathrm{F}}_{cd}{}^b{}_a \tilde\theta^c\wedge\tilde\theta^d
\]
by using Proposition~\ref{connection} and the equations \eqref{d-eta} and \eqref{diff-coframe}.
Since $\Omega_a{}^b$ satisfies the same symmetries as in \eqref{metric-connection},
we present only
$\Omega_0{}^0$, $\Omega_1{}^0$, $\Omega_2{}^0$, $\Omega_0{}^{1}$, $\Omega_2{}^1$, $\Omega_0{}^2$, $\Omega_2{}^2$, $\Omega_1{}^2$, $\Omega_2{}^{\ol2}$.

\begin{Proposition}
The curvature forms of $g^{\mathrm{F}}$ are given by
\begin{gather*}
\Omega_0{}^0=\frac{1}{2}P\bigl(\ol\zeta, \eta\bigr)\theta\wedge\tilde\theta^{1}+\frac{1}{2}P(\zeta, \eta)\theta\wedge\tilde\theta^{\ol1}, \\
\Omega_1{}^0=\frac{1}{2}P\bigl(\ol\zeta, \ol\zeta\bigr)\theta\wedge\tilde\theta^1
 +\frac{1}{2}P\bigl(\zeta, \ol\zeta\bigr)\theta\wedge\tilde\theta^{\ol1}-\frac{1}{4}\tilde\theta^{\ol2}\wedge\tilde\theta^3, \\
 \Omega_2{}^0=\frac{1}{4}\theta\wedge\tilde\theta^{\ol2}-\frac{1}{4}\tilde\theta^{\ol1}\wedge\tilde\theta^3, \\
\Omega_0{}^1=-\frac{1}{2}P(\eta, \eta)\theta\wedge\tilde\theta^1+\frac{1}{2}P(\zeta, \eta)\tilde\theta^1\wedge\tilde\theta^{\ol1}+\frac{1}{4}\tilde\theta^2\wedge\tilde\theta^3, \\
\Omega_2{}^1=\frac{1}{4}\tilde\theta^1\wedge\tilde\theta^{\ol2}+\frac{1}{4}\tilde\theta^2\wedge\tilde\theta^{\ol1}, \\
\Omega_0{}^2=(*C)\bigl(\zeta, \ol\zeta\bigr)\theta\wedge\tilde\theta^1-\frac{1}{2}P(\eta, \eta)\theta\wedge\tilde\theta^2
 -(*C)(\zeta, \zeta)\theta\wedge\tilde\theta^{\ol1}+\frac{1}{2}P(\zeta, \eta)\theta\wedge\tilde\theta^3 \\
\hphantom{\Omega_0{}^2=}{}\hspace{0.2mm}
 -\frac{1}{2}P\bigl(\ol\zeta, \eta\bigr)\tilde\theta^1\wedge\tilde\theta^2
+(*C)(\zeta, \eta)\tilde\theta^1\wedge\tilde\theta^{\ol1}
+\frac{1}{2}P\bigl(\zeta, \ol\zeta\bigr) \tilde\theta^1\wedge\tilde\theta^3
+\frac{1}{2}P(\zeta, \eta)\tilde\theta^2\wedge\tilde\theta^{\ol1}
 \\
\hphantom{\Omega_0{}^2=}{}\hspace{0.2mm}
 +\frac{1}{2}P(\zeta, \zeta)\tilde\theta^{\ol1}\wedge\tilde\theta^3, \\
\Omega_2{}^2=\frac{1}{2}P(\zeta, \eta)\theta\wedge\tilde\theta^{\ol1}+\frac{1}{2}P\bigl(\zeta, \ol\zeta\bigr)\tilde\theta^1\wedge\tilde\theta^{\ol1}+\frac{1}{4}\tilde\theta^2\wedge\tilde\theta^{\ol2}, \\
\Omega_1{}^2=-(*C)\bigl(\ol\zeta, \eta\bigr)\theta\wedge\tilde\theta^1 -\frac{1}{2}P\bigl(\ol\zeta, \eta\bigr)\theta\wedge\tilde\theta^2
+(*C)(\zeta, \eta)\theta\wedge\tilde\theta^{\ol1}
+\frac{1}{2}P(\zeta, \eta)\theta\wedge\tilde\theta^{\ol2}
 \\
\hphantom{\Omega_1{}^2=}{}\hspace{0.2mm}
 -\frac{1}{2}P(\ol\zeta ,\ol\zeta)\tilde\theta^1\wedge\tilde\theta^2
-(*C)(\eta, \eta)\tilde\theta^1\wedge\tilde\theta^{\ol1}
+\frac{1}{2}
P\bigl(\zeta, \ol\zeta\bigr)\tilde\theta^1\wedge\tilde\theta^{\ol2}
+\frac{1}{2}P\bigl(\zeta, \ol\zeta\bigr)\tilde\theta^2\wedge\tilde\theta^{\ol1} \\
\hphantom{\Omega_1{}^2=}{}\hspace{0.2mm}
+\frac{1}{2}P(\zeta, \zeta)\tilde\theta^{\ol1}\wedge\tilde\theta^{\ol2}, \\
\Omega_2{}^{\ol2}=\frac{1}{2}P\bigl(\ol\zeta, \eta\bigr)\theta\wedge\tilde\theta^{\ol1}+\frac{1}{2}P\bigl(\ol\zeta, \ol\zeta\bigr)\tilde\theta^1\wedge\tilde\theta^{\ol1}.
\end{gather*}
The other components are computed by the symmetry and the reality of the curvature form.
\end{Proposition}
It follows that the non-zero components of the curvature tensor are
\begin{alignat*}{4}
&R^{\mathrm{F}}_{010\ol2}=\frac{1}{2}P(\eta, \eta),
\qquad&&
R^{\mathrm{F}}_{0\ol21\ol1}=-\frac{1}{2}P(\zeta, \eta),
\qquad&&
R^{\mathrm{F}}_{0\ol223}=-\frac{1}{4},& \\
&R^{\mathrm{F}}_{010\ol1}=-(*C)\bigl(\zeta, \ol\zeta\bigr),
\qquad&&
R^{\mathrm{F}}_{0\ol10\ol1}=(*C)(\zeta, \zeta),
\qquad&&
R^{\mathrm{F}}_{0\ol103}=-\frac{1}{2}P(\zeta, \eta),& \\
&R^{\mathrm{F}}_{0\ol112}=\frac{1}{2}P\bigl(\ol\zeta, \eta\bigr),
\qquad&&
R^{\mathrm{F}}_{0\ol12\ol1}=-\frac{1}{2}P(\zeta, \eta),
\qquad&&
R^{\mathrm{F}}_{0\ol11\ol1}=-(*C)(\zeta, \eta),& \\
&R^{\mathrm{F}}_{0\ol113}=-\frac{1}{2}P\bigl(\zeta, \ol\zeta\bigr),
\qquad&&
R^{\mathrm{F}}_{0\ol1\ol13}=-\frac{1}{2}P(\zeta, \zeta),
\qquad&&
R^{\mathrm{F}}_{121\ol1}=\frac{1}{2}P\bigl(\ol\zeta, \ol\zeta\bigr),& \\
&R^{\mathrm{F}}_{2\ol12\ol2}=-\frac{1}{4},
\qquad&&
R^{\mathrm{F}}_{1\ol12\ol1}=-\frac{1}{2}P\bigl(\zeta, \ol\zeta\bigr),
\qquad&&
R^{\mathrm{F}}_{23\ol13}=\frac{1}{4},& \\
&R^{\mathrm{F}}_{1\ol11\ol1}=(*C)(\eta, \eta)&
\end{alignat*}
and the components which can be obtained from these by the symmetry
$R^{\mathrm{F}}_{abcd}=R^{\mathrm{F}}_{[ab][cd]}=R^{\mathrm{F}}_{cdab}$ and reality of the curvature tensor.

The non-zero components of the Ricci tensor
\[
R^{\mathrm{F}}_{ab}=R^{\mathrm{F}}_{1a\ol2b}+R^{\mathrm{F}}_{\ol2a1b}+R^{\mathrm{F}}_{2a\ol1b}+R^{\mathrm{F}}_{\ol1a2b}+R^{\mathrm{F}}_{0a3b}+R^{\mathrm{F}}_{3a0b}
\]
are given by
\begin{alignat*}{4}
&R^{\mathrm{F}}_{00}=2P(\eta, \eta), \qquad&&
R^{\mathrm{F}}_{01}=2P\bigl(\ol\zeta, \eta\bigr), \qquad&&
R^{\mathrm{F}}_{2\ol2}=1,& \\
&R^{\mathrm{F}}_{11}=2P\bigl(\ol\zeta, \ol\zeta\bigr),\qquad&&
R^{\mathrm{F}}_{1\ol1}=2P\bigl(\zeta, \ol\zeta\bigr),\qquad&&
R^{\mathrm{F}}_{33}=1.&
\end{alignat*}
It follows that the scalar curvature vanishes,
\[
R^{\mathrm{F}}=2\bigl(R^{\mathrm{F}}_{1\ol2}+R^{\mathrm{F}}_{2\ol1}+R^{\mathrm{F}}_{03}\bigr)=0.
\]
The Schouten tensor is equal to $P^{\mathrm{F}}_{ab}=\frac{1}{4}R^{\mathrm{F}}_{ab}$, so the non-zero components of the Weyl curvature
\[
W^{\mathrm{F}}_{abcd}=R^{\mathrm{F}}_{abcd}-\frac{1}{4}
\bigl(R^{\mathrm{F}}_{ac}g^{\mathrm{F}}_{bd}-R^{\mathrm{F}}_{bc}g^{\mathrm{F}}_{ad}+R^{\mathrm{F}}_{bd}g^{\mathrm{F}}_{ac}-R^{\mathrm{F}}_{ad}g^{\mathrm{F}}_{bc}\bigr)
\]
are given by
\begin{alignat*}{3}
&W^{\mathrm{F}}_{010\ol1}=-(*C)\bigl(\zeta, \ol\zeta\bigr),\qquad&&
W^{\mathrm{F}}_{0\ol10\ol1}=(*C)(\zeta, \zeta),& \\
&W^{\mathrm{F}}_{0\ol11\ol1}=-(*C)(\zeta, \eta),\qquad&&
W^{\mathrm{F}}_{1\ol11\ol1}=(*C)(\eta, \eta)&
\end{alignat*}
and the components which can be obtained from these by the symmetry and reality of the Weyl tensor.
Since $\bigl(\zeta, \ol\zeta, \eta\bigr)$ gives a basis of $\mathbb{C}T\Sigma$ at the base point, we obtain the following theorem.

\begin{Theorem}[{\cite[Theorem 6.7]{Low}} and \cite{SY}]\label{flat}
The Fefferman space $\bigl(\wh{\mathcal{C}}, g^{\mathrm{F}}\bigr)$ is conformally flat if and only if $(\Sigma, [g])$ is conformally flat.
\end{Theorem}

This also implies that the twistor CR manifold $M$ is locally CR equivalent to the hyperquadric if and only if $(\Sigma, [g])$ is conformally flat.

\subsection[Null geodesics of g\^{}F]{Null geodesics of $\boldsymbol{g^{\mathrm{F}}}$}
To examine the projections of chains and null chains to $\Sigma$, we
compute the null geodesic equation for the Fefferman metric.
We fix a representative metric $g\in[g]$, and let \smash{$\gamma(t)=\bigl(x^i(t), \zeta_i(t)\bigr)$} be a curve on $\wh{\mathcal{C}}$ such that the projection \smash{$x(t)\!=\!\bigl(x^i(t)\bigr)$} is a regular curve on $\Sigma$. In the adapted coframe
$\smash{\bigl(\tilde\theta^0\!=\!\theta, \tilde\theta^1, \tilde\theta^2, \tilde\theta^{\ol1}, \tilde\theta^{\ol2}, \tilde\theta^3\bigr)}$, the components \smash{$\dot\gamma^a:=\tilde\th^a(\dot\gamma)$} are given by
\begin{align*}
&\dot\gamma^0=\eta_i \dot x^i=g(\eta, \dot x), \\
&\dot\gamma^1=\zeta_i \dot x^i=g(\zeta, \dot x), \\
&\dot\gamma^2={\rm i}\eta^j\bigl(\dot\zeta_j-\Gamma_{kj}{}^l\zeta_l \dot x^k\bigr)
={\rm i}g(\eta, \nabla_{\dot x}\zeta), \\
&\dot\gamma^3=-\frac{{\rm i}}{2}\bigl(
\ol{\zeta^j}\bigl(\dot\zeta_j-\Gamma_{kj}{}^l\zeta_l \dot x^k\bigr)
-\zeta^j\bigl(\ol{\dot\zeta_j}-\Gamma_{kj}{}^l\ol{\zeta_l} \dot x^k\bigr)
\bigr)=\operatorname{Im} g\bigl(\ol\zeta, \nabla_{\dot x}\zeta\bigr)
=-{\rm i} g\bigl(\ol\zeta, \nabla_{\dot x}\zeta\bigr),
\end{align*}
where $\nabla$ is the Levi-Civita connection of $g$.
The nullity of $\gamma$ is written as
\begin{equation}\label{nullity}
g^{\rm F}(\dot\gamma, \dot\gamma)=2\bigl(\dot\gamma^1\dot\gamma^{\ol2}+\dot\gamma^2\dot\gamma^{\ol1}\bigr)+2\dot\gamma^0\dot\gamma^3=0.
\end{equation}
Let $\nabla^{\mathrm{F}}$ be the Levi-Civita connection of $g^{\mathrm{F}}$. By Proposition~\ref{connection}, we can compute the acceleration
\[
\nabla^{\mathrm{F}}_{\dot\gamma}\dot\gamma^a
=\ddot\gamma^a+\omega_0{}^a(\dot\gamma)\dot\gamma^0
+\omega_1{}^a(\dot\gamma)\dot\gamma^1+\omega_2{}^a(\dot\gamma)\dot\gamma^2+\omega_{\ol1}{}^a(\dot\gamma)\dot\gamma^{\ol1}+\omega_{\ol2}{}^a(\dot\gamma)\dot\gamma^{\ol2}+\omega_3{}^a(\dot\gamma)\dot\gamma^3
\]
as follows:
\begin{gather*}
\nabla^{\mathrm{F}}_{\dot\gamma}\dot\gamma^0=\ddot\gamma^0, \\
\nabla^{\mathrm{F}}_{\dot\gamma}\dot\gamma^1=\ddot\gamma^1, \\
\nabla^{\mathrm{F}}_{\dot\gamma}\dot\gamma^2
= \ddot\gamma^2+{\rm i}P(\zeta, \eta)\bigl(\dot\gamma^0\bigr)^2+{\rm i}\bigl(P\bigl(\zeta, \ol\zeta\bigr)-P(\eta, \eta)\bigr)\dot\gamma^0\dot\gamma^1+{\rm i}P(\zeta, \zeta)\dot\gamma^0\dot\gamma^{\ol1} \\
\hphantom{\nabla^{\mathrm{F}}_{\dot\gamma}\dot\gamma^2=}{}
-{\rm i}P\bigl(\ol\zeta, \eta\bigr)\bigl(\dot\gamma^1\bigr)^2-{\rm i}P(\zeta, \eta)|\dot\gamma^1|^2, \\
\nabla^{\mathrm{F}}_{\dot\gamma}\dot\gamma^3
=\ddot\gamma^3+{\rm i}P\bigl(\ol\zeta, \eta\bigr)\dot\gamma^0\dot\gamma^1-{\rm i}P(\zeta, \eta)\dot\gamma^0\dot\gamma^{\ol1}+{\rm i}P\bigl(\ol\zeta, \ol\zeta\bigr)\bigl(\dot\gamma^1\bigr)^2-{\rm i}P(\zeta, \zeta)\bigl(\dot\gamma^{\ol1}\bigr)^2.
\end{gather*}
Hence, $\gamma$ is a geodesic of $g^{\mathrm{F}}$ if and only if
\begin{equation*}
\dot\gamma^0=\mathrm{const}, \qquad \dot\gamma^1=\mathrm{const},
\end{equation*}
and
\begin{gather}
 \ddot\gamma^2+{\rm i}P(\zeta, \eta)\bigl(\dot\gamma^0\bigr)^2+{\rm i}\bigl(P\bigl(\zeta, \ol\zeta\bigr)-P(\eta, \eta)\bigr)\dot\gamma^0\dot\gamma^1+{\rm i}P(\zeta, \zeta) \dot\gamma^0\dot\gamma^{\ol1} \nonumber \\
 \qquad{} -{\rm i}P\bigl(\ol\zeta, \eta\bigr)\bigl(\dot\gamma^1\bigr)^2-{\rm i}P(\zeta, \eta)|\dot\gamma^1|^2=0,\label{geodesic2}
\\
\ddot\gamma^3+{\rm i}P\bigl(\ol\zeta, \eta\bigr)\dot\gamma^0\dot\gamma^1-{\rm i}P(\zeta, \eta)\dot\gamma^0\dot\gamma^{\ol1}+{\rm i}P\bigl(\ol\zeta, \ol\zeta\bigr)\bigl(\dot\gamma^1\bigr)^2-{\rm i}P(\zeta, \zeta)\bigl(\dot\gamma^{\ol1}\bigr)^2=0.\label{geodesic3}
\end{gather}

We will derive the equation satisfied by $x(t)$ from the null geodesic equation for $\gamma(t)$. We note that for any $V, W\in \mathbb{C}T_x\Sigma$, the identities
\begin{align*}
&V=g(V, \eta)\eta+g\bigl(V, \ol\zeta\bigr)\zeta+g(V, \zeta)\ol\zeta, \\
&g(V, W)=g(V, \eta)g(W, \eta)+g(V, \zeta)g\bigl(W, \ol\zeta\bigr)+g\bigl(V, \ol\zeta\bigr)g(W, \zeta)
\end{align*}
hold on $\hat\pi^{-1}(x)\subset\wh{\mathcal{C}}$. In particular, we have
\begin{equation}\label{x-zeta-eta}
\dot x=\dot\gamma^0\eta+\dot\gamma^{\ol1}\zeta+\dot\gamma^1\ol\zeta, \qquad |\dot x|^2=\bigl(\dot\gamma^0\bigr)^2+2\bigl|\dot\gamma^1\bigr|^2.
\end{equation}
\begin{Lemma}
For any curve $\gamma(t)=(x(t), \zeta(t))$ on $\wh{\mathcal{C}}$, it holds that
\begin{alignat}{3}
&g\bigl(\nabla_{\dot x}\nabla_{\dot x}\zeta, \ol\zeta\bigr) ={\rm i}\ddot\gamma^3-\bigl|\dot\gamma^2\bigr|^2-\bigl(\dot\gamma^3\bigr)^2, \qquad&&
g(\nabla_{\dot x}\nabla_{\dot x}\zeta, \zeta) =\bigl(\dot\gamma^2\bigr)^2,& \nonumber \\
&g(\nabla_{\dot x}\nabla_{\dot x}\zeta, \eta) =-{\rm i}\ddot\gamma^2+\dot\gamma^2\dot\gamma^3, \qquad&&
g(\nabla_{\dot x}\nabla_{\dot x}\eta, \zeta) ={\rm i}\ddot\gamma^2+\dot\gamma^2\dot\gamma^3, &\nonumber \\
&g(\nabla_{\dot x}\nabla_{\dot x}\eta, \eta) =-2\bigl|\dot\gamma^2\bigr|^2.&\label{nabla-zeta-eta}
\end{alignat}
\end{Lemma}
\begin{proof}
We only prove the identity for $g(\nabla_{\dot x}\nabla_{\dot x}\eta, \zeta)$ since the other cases are similar. Using the identities noted above, we have
\begin{align*}
g(\nabla_{\dot x}\nabla_{\dot x}\eta, \zeta)
&=\frac{{\rm d}}{{\rm d}t}g(\nabla_{\dot x}\eta, \zeta)-g(\nabla_{\dot x}\eta, \nabla_{\dot x}\zeta) =-\frac{{\rm d}}{{\rm d}t}g(\eta, \nabla_{\dot x}\zeta)-g(\nabla_{\dot x}\eta, \zeta)g\bigl(\nabla_{\dot x}\zeta, \ol\zeta\bigr) \\
&=-\frac{{\rm d}}{{\rm d}t}g(\eta, \nabla_{\dot x}\zeta)+g(\eta, \nabla_{\dot x}\zeta)g\bigl(\nabla_{\dot x}\zeta, \ol\zeta\bigr)
={\rm i}\ddot\gamma^2+\dot\gamma^2\dot\gamma^3.
\tag*{\qed}
\end{align*}
\renewcommand{\qed}{}
\end{proof}

Now we are ready to prove the following.

\begin{Theorem}\label{null-geod-proj}
Let $g\in[g]$ be a representative metric and let $\gamma(t)=(x(t), \zeta(t))$ be a null geodesic of $g^{\mathrm{F}}$ on $\wh{\mathcal{C}}$ such that the projection $x(t)$ is a regular curve on $\Sigma$. Then, $|\dot x|_g$ is constant, and $x(t)$ satisfies
\begin{equation}\label{eq-x}
\nabla_{\dot x}\nabla_{\dot x}\dot x-|\dot x|^2P(\dot x)=
-\bigl(2\bigl|\dot\gamma^2\bigr|^2+\bigl(\dot\gamma^3\bigr)^2+P(\dot x, \dot x)\bigr)\dot x.
\end{equation}
In particular, $x(t)$ is a conformal geodesic of $(\Sigma, [g])$.
\end{Theorem}
\begin{proof}
We prove that the inner products of both sides of \eqref{eq-x} with $\zeta$, $\eta$ coincide.

Since $\gamma$ is a geodesic, $\dot\gamma^0$ and $\dot\gamma^1$ are constant.
Hence, by \eqref{x-zeta-eta}, $|\dot x|$ is constant and we have
\begin{align*}
\nabla_{\dot x}\nabla_{\dot x}\dot x-|\dot x|^2P(\dot x)
=&\dot\gamma^0\nabla_{\dot x}\nabla_{\dot x}\eta+
\dot\gamma^{\ol1}\nabla_{\dot x}\nabla_{\dot x}\zeta+\dot\gamma^1\nabla_{\dot x}\nabla_{\dot x}\ol\zeta \\
&{}{-}\,\bigl(\bigl(\dot\gamma^0\bigr)^2+2\bigl|\dot\gamma^1\bigr|^2\bigr)\bigl(\dot\gamma^0P(\eta)+\dot\gamma^{\ol1}P(\zeta)+\dot\gamma^1P(\ol\zeta)\bigr).
\end{align*}
By using \eqref{nabla-zeta-eta}, we compute as
\begin{align*}
g\bigl(\nabla_{\dot x}\nabla_{\dot x}\dot x-|\dot x|^2P(\dot x) ,\zeta\bigr)
={}&\dot\gamma^0g(\nabla_{\dot x}\nabla_{\dot x}\eta, \zeta)+
\dot\gamma^{\ol1}g(\nabla_{\dot x}\nabla_{\dot x}\zeta, \zeta)+\dot\gamma^1g\bigl(\nabla_{\dot x}\nabla_{\dot x}\ol\zeta, \zeta\bigr) \\
&{}{-}\,\bigl(\bigl(\dot\gamma^0\bigr)^2+2\bigl|\dot\gamma^1\bigr|^2\bigr)\bigl(\dot\gamma^0P(\eta, \zeta)+\dot\gamma^{\ol1}P(\zeta, \zeta)+\dot\gamma^1P\bigl(\ol\zeta, \zeta\bigr)\bigr) \\
={}&{\rm i}\dot\gamma^0\ddot\gamma^2+\dot\gamma^0\dot\gamma^2\dot\gamma^3+\dot\gamma^{\ol1}\bigl(\dot\gamma^2\bigr)^2
-{\rm i}\dot\gamma^1\ddot\gamma^3-\dot\gamma^1\bigl|\dot\gamma^2\bigr|^2-\dot\gamma^1\bigl(\dot\gamma^3\bigr)^2 \\
&{}{-}\bigl(\bigl(\dot\gamma^0\bigr)^2+2|\dot\gamma^1|^2\bigr)\bigl(\dot\gamma^0P(\zeta, \eta)+\dot\gamma^{\ol1}P(\zeta, \zeta)+\dot\gamma^1P\bigl(\zeta, \ol\zeta\bigr)\bigr).
\end{align*}
By the nullity \eqref{nullity}, we have
\[
\dot\gamma^0\dot\gamma^2\dot\gamma^3+\dot\gamma^{\ol1}\bigl(\dot\gamma^2\bigr)^2 -\dot\gamma^1\bigl|\dot\gamma^2\bigr|^2=-2\dot\gamma^1\bigl|\dot\gamma^2\bigr|^2.
\]
Moreover, the geodesic equations \eqref{geodesic2} and \eqref{geodesic3} give
\begin{gather*}
{\rm i}\dot\gamma^0\ddot\gamma^2=P(\zeta, \eta)\bigl(\dot\gamma^0\bigr)^3+\bigl(P\bigl(\zeta, \ol\zeta\bigr)-P(\eta, \eta)\bigr)\bigl(\dot\gamma^0\bigr)^2\dot\gamma^1+P(\zeta, \zeta) \bigl(\dot\gamma^0\bigr)^2\dot\gamma^{\ol1} \\
\hphantom{{\rm i}\dot\gamma^0\ddot\gamma^2=}{}
-P\bigl(\ol\zeta, \eta\bigr)\dot\gamma^0\bigl(\dot\gamma^1\bigr)^2-P(\zeta, \eta)\dot\gamma^0\bigl|\dot\gamma^1\bigr|^2, \\
-{\rm i}\dot\gamma^1\ddot\gamma^3=-P\bigl(\ol\zeta, \eta\bigr)\dot\gamma^0\bigl(\dot\gamma^1\bigr)^2+P(\zeta, \eta)\dot\gamma^0\bigl|\dot\gamma^{1}\bigr|^2
-P\bigl(\ol\zeta, \ol\zeta\bigr)\bigl(\dot\gamma^1\bigr)^3+P(\zeta, \zeta)\dot\gamma^{\ol1}\bigl|\dot\gamma^{1}\bigr|^2.
\end{gather*}
Substituting these expressions, we obtain
\begin{align*}
g\bigl(\nabla_{\dot x}\nabla_{\dot x}\dot x-|\dot x|^2P(\dot x) ,\zeta\bigr)
&=-2\dot\gamma^1\bigl|\dot\gamma^2\bigr|^2-\dot\gamma^1\bigl(\dot\gamma^3\bigr)^2-\dot\gamma^1P(\dot x, \dot x) \\
&=-\bigl(2\bigl|\dot\gamma^2\bigr|^2+\bigl(\dot\gamma^3\bigr)^2+P(\dot x, \dot x)\bigr)g(\dot x, \zeta).
\end{align*}

Similarly, we compute as
\begin{align*}
g\bigl(\nabla_{\dot x}\nabla_{\dot x}\dot x-|\dot x|^2P(\dot x) ,\eta\bigr)
={}&\dot\gamma^0g(\nabla_{\dot x}\nabla_{\dot x}\eta, \eta)+
\dot\gamma^{\ol1}g(\nabla_{\dot x}\nabla_{\dot x}\zeta, \eta)+\dot\gamma^1g\bigl(\nabla_{\dot x}\nabla_{\dot x}\ol\zeta, \eta\bigr) \\
&{}{-}\bigl(\bigl(\dot\gamma^0\bigr)^2+2\bigl|\dot\gamma^1\bigr|^2\bigr)\bigl(\dot\gamma^0P(\eta, \eta)+\dot\gamma^{\ol1}P(\zeta, \eta)+\dot\gamma^1P\bigl(\ol\zeta, \eta\bigr)\bigr) \\
={}&{-}2\dot\gamma^0\bigl|\dot\gamma^2\bigr|^2-{\rm i}\dot\gamma^{\ol1}\ddot\gamma^2
+\dot\gamma^{\ol1}\dot\gamma^2\dot\gamma^3+{\rm i}\dot\gamma^1\ddot\gamma^{\ol2}+\dot\gamma^1\dot\gamma^{\ol2}\dot\gamma^3 \\
&{}{-}\bigl(\bigl(\dot\gamma^0\bigr)^2+2\bigl|\dot\gamma^1\bigr|^2\bigr)\bigl(\dot\gamma^0P(\eta, \eta)+\dot\gamma^{\ol1}P(\zeta, \eta)+\dot\gamma^1P\bigl(\ol\zeta, \eta\bigr)\bigr).
\end{align*}
By the nullity and the equation \eqref{geodesic2}, we have
\begin{align*}
\dot\gamma^1\dot\gamma^{\ol2}\dot\gamma^3+\dot\gamma^{\ol1}\dot\gamma^2\dot\gamma^3 =-\dot\gamma^0\bigl(\dot\gamma^3\bigr)^2
\end{align*}
and
\begin{align*}
-{\rm i}\dot\gamma^{\ol1}\ddot\gamma^2+{\rm i}\dot\gamma^1\ddot\gamma^{\ol2}
={}&{-}P(\zeta, \eta)\bigl(\dot\gamma^0\bigr)^2\dot\gamma^{\ol1} -P\bigl(\ol\zeta, \eta\bigr)\bigl(\dot\gamma^0\bigr)^2\dot\gamma^1-2\bigl(P\bigl(\zeta, \ol\zeta\bigr)-P(\eta, \eta)\bigr)
\dot\gamma^0\bigl|\dot\gamma^1\bigr|^2\\
&{}{-}\, P(\zeta, \zeta) \dot\gamma^0\bigl(\dot\gamma^{\ol1}\bigr)^2\!-P\bigl(\ol\zeta, \ol\zeta\bigr) \dot\gamma^0\bigl(\dot\gamma^{1}\bigr)^2+2P\bigl(\ol\zeta, \eta\bigr)\dot\gamma^1\bigl|\dot\gamma^1\bigr|^2
 + 2P(\zeta, \eta)\dot\gamma^{\ol1}\bigl|\dot\gamma^1\bigr|^2.
\end{align*}
Substituting these expressions, we obtain
\begin{align*}
g\bigl(\nabla_{\dot x}\nabla_{\dot x}\dot x-|\dot x|^2P(\dot x) ,\eta\bigr)
&=-2\dot\gamma^0\bigl|\dot\gamma^2\bigr|^2-\dot\gamma^0\bigl(\dot\gamma^3\bigr)^2-\dot\gamma^0P(\dot x, \dot x) \\
&=-\bigl(2\bigl|\dot\gamma^2\bigr|^2+\bigl(\dot\gamma^3\bigr)^2+P(\dot x, \dot x)\bigr)g(\dot x, \eta).
\end{align*}
Thus, $x(t)$ satisfies \eqref{eq-x}. In particular, it satisfies the conformal geodesic equation \eqref{conf-geod-const-speed} for curves with constant speed.
\end{proof}

By definition, a (null) chain $(x(t), [\zeta(t)])$ is the projection of a non-vertical null geodesic $\gamma(t)=(x(t), \zeta(t))$ on $\wh{\mathcal{C}}$ to $M$. If the constants $\dot\gamma^0=g(\eta, \dot x)$,
$\dot\gamma^1=g(\zeta, \dot x)$ are both equal to~0, then we have $\dot x=0$ and the projection $x(t)$ is a~constant curve on $\Sigma$. Otherwise, it is a regular curve and hence a conformal geodesic on $\Sigma$ by the above theorem. Thus we obtain Theorem~\ref{chain-projection}\,(1).

As an application, we have the following.

\begin{Theorem}[cf.\ \cite{SY}]
Let $M_i$, $i=1, 2$ be the twistor CR manifolds over conformal $3$-manifolds $(\Sigma_i, [g_i])$, $i=1, 2$. If $\wh f\colon M_1\to M_2$ is a bundle isomorphism which is a CR equivalent map, then the underlying map $f\colon (\Sigma_1, [g_1])\to (\Sigma_2, [g_2])$ is a conformal isomorphism.
\end{Theorem}
\begin{proof}
Since $\wh f$ maps any (null) chain on $M_1$ to a (null) chain on $M_2$, $f$ maps conformal geodesics on $\Sigma_1$ to conformal geodesics on $\Sigma_2$. Hence, by~\cite[Theorem 2.3]{Kuo}, $f$ is a conformal isomorphism.
\end{proof}

\subsection{Canonical lift of conformal geodesics}\label{lift}
Theorem~\ref{chain-projection}\,(2) asserts that we can lift a conformal geodesic $x(t)$ to a unique (null) chain on $M$ once we choose a lift of the endpoint $x(a)$; we prove this in Section~\ref{sphere-bundle} below.
Here, as a special case, we present a canonical lift of a conformal geodesic to a chain.

Given a regular curve $x(t)$ on $\Sigma$, we take a complex vector field $\zeta(t)$ along $x(t)$ such that $(\sqrt{2}\operatorname{Re}\zeta, \sqrt{2}\operatorname{Im}\zeta)$ gives an oriented orthonormal basis of $\dot x^\perp$ with respect to a metric ${g\in[g]}$. Such a $\zeta(t)$ is determined by the conformal structure $[g]$ up to $\mathbb{C}^*$-action, so we have a canonical~lift
\[
\tilde x(t):=(x(t), [\zeta(t)])
\]
on $M$. Note that $\tilde x(t)$ is transverse to the contact distribution since
$\th(\dot{\tilde x})=|\dot x|^2\neq 0$.
\begin{Theorem} \label{canonical-lift}
The curve $\tilde x(t)$ is a chain if and only if $x(t)$ is a conformal geodesic. Moreover, for a conformal geodesic $x(t)$, $a\le t\le b$, the curve $\tilde x(t)$ is a unique chain through $(x(a), [\zeta(a)])$ which projects to $x(t)$.
\end{Theorem}
\begin{proof}
If $\tilde x(t)$ is a chain, then $x(t)$ is a conformal geodesic by Theorem~\ref{chain-projection}\,(1). Conversely, suppose $x(t)$ is a conformal geodesic. By the characterization of conformal geodesics given by Proposition~\ref{conf-geod-characterization}, we may take a representative metric $g\in[g]$ and a parametrization $t$ such that $\nabla_{\dot x}\dot x=P(\dot x)=0$ and $|\dot x|=1$. Since $\dot x$ is parallel, we can take
$\zeta(t)$ as a parallel vector field along~$x(t)$.
Then, $\gamma(t):=(x(t), \zeta(t))$ becomes a curve on $\wh{\mathcal{C}}$ with
\[
\eta=\dot x, \qquad \dot\gamma^0=1, \qquad \dot\gamma^1=\dot\gamma^2=\dot\gamma^3=0.
\]
In particular, it is a null curve. Noting that $P(\eta)=0$, we see that $\gamma(t)$ satisfies the geodesic equations \eqref{geodesic2} and \eqref{geodesic3}, and hence $\tilde x(t)=(x(t), [\zeta(t)])$ is a chain on $M$.

Next, we prove the uniqueness. Suppose that $c (t)$ is a chain through $(x(a), [\zeta(a)])$ which projects to a conformal geodesic $x(t)$, and let $\gamma_c(t)=(x(t), \zeta_c(t))$ be a null geodesic on~${\bigl(\wh{\mathcal{C}}, g^{\mathrm{F}}\bigr)}$, the Fefferman space determined by $g$ above, which projects to $c(t)$. Then, $\dot\gamma_c^0=g(\eta_c, \dot x)$ and $\dot\gamma_c^1=g(\zeta_c, \dot x)$ are constant by the geodesic equation. Since $\dot\gamma_c^0(a)=|\dot x(a)|^2=1$, $\dot\gamma_c^1(a)=0$, we obtain that $\eta_c(t)=\dot x(t)$ for any $t$ and hence $c(t)=\tilde x(t)$.
\end{proof}

\section[A variational principle for conformal geodesics in dimension three]{A variational principle for conformal geodesics\\ in dimension three}\label{variational-principle}

\subsection{Relations to the unit tangent sphere bundle and the frame bundle}\label{sphere-bundle}
We will give an alternative description of the previous constructions in terms of the frame bundle over $\Sigma$.
We fix a representative metric $g\in[g]$ and let
$\mathcal{P}$ be the $\mathrm{SO}(3)$-bundle of oriented orthonormal frames over $(\Sigma, g)$. Then, the map
\begin{equation}\label{isom-C-hat}
f\colon\ \wh{\mathcal{C}}\overset{\sim}{\longrightarrow} \mathcal{P}, \qquad
\zeta \longmapsto B=\bigl(\sqrt{2}\operatorname{Re}\zeta, \sqrt{2}\operatorname{Im}\zeta, \eta\bigr)
\end{equation}
gives a bundle isomorphism. The action of $\mathrm{U}(1)$ on $\wh{\mathcal{C}}$ corresponds to rotations of the first two basis vectors in $B$. Thus, the twistor CR manifold $M= \wh{\mathcal{C}}/\mathrm{U}(1)$ is isomorphic to the unit tangent sphere bundle of $(\Sigma, g)$ by the following map:
\begin{equation}\label{M-isom}
M \overset{\sim}{\longrightarrow} S\Sigma, \qquad
[\zeta]\longmapsto \eta.
\end{equation}

Let
\[
\phi=\begin{pmatrix}
\phi^1 \\ \phi^2 \\ \phi^3
\end{pmatrix}
\colon\ T\mathcal{P}\longrightarrow \mathbb{R}^3, \qquad \omega=\begin{pmatrix}
\hphantom{-}0 & \hphantom{-}\omega^3 & -\omega^2 \\
-\omega^3 & \hphantom{-}0 & \hphantom{-}\omega^1 \\
\hphantom{-}\omega^2 & -\omega^1 & \hphantom{-}0
\end{pmatrix}
\colon\ T\mathcal{P}\longrightarrow \mathfrak{so}(3)
\]
be the canonical $1$-form and the Levi-Civita connection form of $g$.
These forms are related to the adapted coframe \smash{$\bigl\{\theta, \tilde\theta^1, \tilde\theta^2, \tilde\theta^{\ol1}, \tilde\theta^{\ol2}, \tilde\theta^3\bigr\}$} on $\wh{\mathcal{C}}$ as follows.
\begin{Lemma}
We have
\[
f^*\phi=
\begin{pmatrix}
\sqrt{2}\operatorname{Re}\tilde\theta^1 \\
\sqrt{2}\operatorname{Im}\tilde\theta^1 \\
\theta
\end{pmatrix}, \qquad
f^*\omega=\begin{pmatrix}
0 & \tilde\theta^3 & -\sqrt{2}\operatorname{Im}\tilde\theta^2 \\
-\tilde\theta^3 & 0 & \hphantom{-}\sqrt{2}\operatorname{Re}\tilde\theta^2 \\
\sqrt{2}\operatorname{Im}\tilde\theta^2 & -\sqrt{2}\operatorname{Re}\tilde\theta^2& 0
\end{pmatrix}.
\]
\end{Lemma}
\begin{proof}
We fix a point $x_0\in\Sigma$ and take normal coordinates $(x^i)$ of $g$ around
$x_0$. Then, at any point $\zeta\in \wh{\mathcal{C}}_{x_0}$, we have
\[
f^*\phi={}^t \!B\, {\rm d}x =\begin{pmatrix}
\sqrt{2}\operatorname{Re}{}^t\zeta \\
\sqrt{2}\operatorname{Im}{}^t\zeta \\
{}^t\eta
\end{pmatrix}
\begin{pmatrix}
{\rm d}x^1 \\ {\rm d}x^2 \\ {\rm d}x^3
\end{pmatrix}
\]
and
\[
f^*\omega={}^t\!B\,{\rm d}B=
\begin{pmatrix}
\sqrt{2}\operatorname{Re}{}^t\zeta \\
\sqrt{2}\operatorname{Im}{}^t\zeta \\
{}^t\eta
\end{pmatrix}
\begin{pmatrix}
\sqrt{2}\operatorname{Re}{\rm d}\zeta & \sqrt{2}\operatorname{Im} {\rm d}\zeta & {\rm d}\eta
\end{pmatrix},
\]
where we write $\zeta$, $\eta$ as column vectors. Thus we obtain the desired formulas.
\end{proof}

We define the Fefferman metric $G^{\mathrm{F}}$ on $\mathcal{P}$ via the isomorphism \eqref{isom-C-hat}: $g^{\mathrm{F}}=f^* G^{\mathrm{F}}$. Then, the above lemma and Theorem~\ref{Fefferman-metric} prove the following.

\begin{Proposition}
The Fefferman metric $G^{\mathrm{F}}$ on $\mathcal{P}$ is given by
\[
G^{\mathrm{F}}=2\bigl( \phi^1\cdot\omega^1+ \phi^2\cdot\omega^2+ \phi^3\cdot\omega^3\bigr).
\]
\end{Proposition}

\begin{Remark}\quad
\begin{itemize}\itemsep=0pt
\item[(i)]
If we denote by $R_A$ the right action of $A\in\mathrm{SO}(3)$ on $\mathcal{P}$, we have
\[
R_A^*\begin{pmatrix}
\phi^1 \\ \phi^2 \\ \phi^3
\end{pmatrix}=A^{-1}\begin{pmatrix}
\phi^1 \\ \phi^2 \\ \phi^3
\end{pmatrix}, \qquad R_A^*\begin{pmatrix}
\omega^1 \\ \omega^2 \\ \omega^3
\end{pmatrix}=A^{-1}\begin{pmatrix}
\omega^1 \\ \omega^2 \\ \omega^3
\end{pmatrix}.
\]
Hence, $G^{\mathrm{F}}$ is a right $\mathrm{SO}(3)$-invariant metric: $R_A^*G^{\mathrm{F}}=G^{\mathrm{F}}$.

\item[(ii)] Since the conformal class of the Fefferman metric is invariantly defined from the CR structure, the conformal manifold $\bigl(\mathcal{P}, \bigl[G^{\mathrm{F}}\bigr]\bigr)$ should also be canonically associated to $(\Sigma, [g])$. This can be checked directly as follows.
Let $\wh g={\rm e}^{2\Upsilon}g$ be another representative metric and~$\wh{\mathcal{P}}$ be the oriented orthonormal frame bundle for $\wh g$.
Then, we have the canonical bundle isomorphism
\[
\d\colon\ \mathcal{P}\overset{\sim}{\longrightarrow}
\wh{\mathcal{P}}, \qquad B\longmapsto {\rm e}^{-\Upsilon}B.
\]
We denote by
$\wh\phi$, $\wh\omega$ the canonical 1-form and the connection form of $\wh g$ on $\wh{\mathcal{P}}$. Then, we have $\d^*\wh\phi={\rm e}^{\Upsilon}\phi$ and
the conformal transformation formula for the Levi-Civita connection gives
\begin{align*}
&\d^*\wh\omega^1=\omega^1+\Upsilon^3\phi^2-\Upsilon^2\phi^3, \\
&\d^*\wh\omega^2=\omega^2+\Upsilon^1\phi^3-\Upsilon^3\phi^1, \\
&\d^*\wh\omega^3=\omega^3+\Upsilon^2\phi^1-\Upsilon^1\phi^2,
\end{align*}
where we set $\Upsilon^i:=\phi^i(\mathrm{grad}_g \Upsilon)\in C^\infty(\mathcal{P})$, which is well-defined since $\phi$ is horizontal.
Thus, the Fefferman metric $\wh{G}^{\mathrm{F}}$ on $\wh{\mathcal{P}}$ satisfies
\[
\d^* \wh{G}^{\mathrm{F}}={\rm e}^\Upsilon G^{\mathrm{F}}.
\]

\item[(iii)] The conformal metric $G^{\mathrm{F}}$ on $\mathcal{P}$ is also considered by Holland~\cite{H}, and it is proved in {\cite[Section~3.5.4]{H}} that null geodesics of this metric project to conformal geodesics on $\Sigma$. This gives another proof to Theorem~\ref{null-geod-proj}.
\end{itemize}
\end{Remark}

In Section~\ref{lift}, we defined a canonical lift of a regular curve $x(t)$ on $\Sigma$ to a curve $\tilde x(t)=(x(t), [\zeta(t)])$ on $M$ which becomes a chain if and only if $x(t)$ is a conformal geodesic. This lift corresponds to the curve $(x(t), \dot x(t)/|\dot x(t)|)$ on $S\Sigma$ via the isomorphism~\eqref{M-isom}. Thus, if we regard $S\Sigma$ as the twistor CR manifold,
Theorem~\ref{canonical-lift} yields the following.

\begin{Theorem}\label{chain-conf-geod2}
For a regular curve $x(t)$ on $(\Sigma, g)$, the curve $\tilde x(t)=(x(t), \dot x(t)/|\dot x(t)|)$ on $S\Sigma$ is a~chain if and only if $x(t)$ is a conformal geodesic of $[g]$. Moreover, for a conformal geodesic $x(t)$, $a\le t\le b$, the curve $\tilde x(t)$ is a unique chain through $(x(a), \dot x(a)/|\dot x(a)|)$ which projects to~$x(t)$.
\end{Theorem}

We can compete the proof of Theorem~\ref{chain-projection}\,(2) by making use of the right $\mathrm{SO}(3)$-invariance of~$G^{\mathrm{F}}$.

\begin{proof}[Proof of Theorem~\ref{chain-projection}\,(2)]
Let $x(t)$, $a\le t\le b$, be a conformal geodesic on $\Sigma$ and $\eta_0\in S_{x(a)}\Sigma$. Take a null geodesic $\gamma(t)$ on $\bigl(\mathcal{P}, G^{\mathrm{F}}\bigr)$ which projects to the chain $\tilde x(t)=(x(t), \dot x(t)/|\dot x(t)|)$.
We choose $A\in \mathrm{SO}(3)$ so that the third basis vector in $R_A(\gamma(a))$ coincides with $\eta_0$, and set $\gamma'(t):=R_A(\gamma(t))$. Since $G^{\mathrm{F}}$ is right-invariant, $\gamma'$ is also a null geodesic. Projecting $\gamma'$ to $S\Sigma$, we obtain a chain ($g(\eta_0, \dot x(a))\neq0$) or a null chain
 ($g(\eta_0, \dot x(a))=0$) through $\eta_0$. The uniqueness also follows from that of $\tilde x(t)$ and the fact that a lift of chain to a null geodesic is uniquely determined by the initial value.
\end{proof}

\subsection[The Kropina metric for chains on S Sigma and a variational characterization of conformal geodesics]{The Kropina metric for chains on $\boldsymbol{S\Sigma}$ and a variational \\ characterization of conformal geodesics}\label{Kropina}

Hereafter, we identify the twistor CR manifold $M$ with the unit tangent sphere bundle $S\Sigma$ for $g\in[g]$, and consider the Fefferman space $\bigl(\mathcal{P}, G^{\mathrm{F}}\bigr)$ instead of $\bigl(\wh{\mathcal{C}}, g^{\mathrm{F}}\bigr)$. We also write $K$ for the infinitesimal generator of the $S^1$-action on $\mathcal{P}$.

To derive a variational principle for conformal geodesics, we first recall from~\cite{CMMM} a characterization of CR chains as geodesics of a Kropina metric.
\begin{Definition}
Let $s\colon U\to\mathcal{P}$ be a local section defined on an open subset $U\subset S\Sigma$. The Finsler metric
\[
F(\tilde x, \xi):=\frac{G^{\mathrm{F}}(s_* \xi, s_*\xi)}{G^{\mathrm{F}}(K, s_*\xi)}, \qquad \tilde x\in U,\quad \xi\in T_{\tilde x}S\Sigma\setminus\operatorname{Ker}\th
\]
is called the {\it Kropina metric} on $U$ associated with $s$.
\end{Definition}
We note that the condition $\xi\notin\operatorname{Ker}\th$ is equivalent to $G^{\mathrm{F}}(K, s_*\xi)\neq0$, and $F$ is not defined on the contact distribution.

\begin{Theorem}[{\cite{CMMM}}]\label{chain-kropina}
A curve $\tilde x(t)$ on $U\subset S\Sigma$ which is transverse to the contact distribution is a chain if and only if it is an $($unparametrized$)$ geodesic of $F$.
\end{Theorem}

For a regular curve $x(t)$, $a\le t\le b$, on $\Sigma$, we consider the Kropina length of the lift
$\tilde x(t):=(x(t), \dot x(t)/|\dot x(t)|)$ associated with a section $s\colon U\to \mathcal{P}$ on a neighborhood $U\subset S\Sigma$ of the curve $\tilde x(t)$, and define the functional
\[
\mathscr{L}[x(t)]:=\int_a^b F\bigl(\tilde x(t), \dot{\tilde x}(t)\bigr)\, {\rm d}t.
\]
Noting that
\[
\phi(s_*\dot{\tilde x})=\begin{pmatrix}
0 \\ 0 \\ |\dot x|
\end{pmatrix}, \qquad \phi(K)=0, \qquad \omega(K)=Y:=\begin{pmatrix}
\hphantom{-}0 & 1 & 0\\
-1 & 0 & 0 \\
\hphantom{-}0& 0& 0
\end{pmatrix},
\]
we have
\[
F\bigl(\tilde x(t), \dot{\tilde x}(t)\bigr)=2\omega^3\bigl(s_*\dot{\tilde x}\bigr).
\]

If $x(t)$ is a conformal geodesic, $\tilde x(t)$ is a chain by Theorem~\ref{chain-conf-geod2}, and hence $x(t)$ is a critical curve of $\mathscr{L}$ by Theorem~\ref{chain-kropina}. To prove the converse, we will show that the variation of the Kropina length in the vertical directions always vanishes.

\begin{Proposition}\label{vertical-variation}
Let $\gamma(t)=(x(t), B(t))$, $a\le t\le b$, be a curve on $\mathcal{P}$ which projects to $\tilde x(t)=(x(t), \dot x(t)/|\dot x(t)|)$. Let $\gamma_\epsilon(t)$ be a variation of $\gamma(t)$ fixing the endpoints such that the variation vector field $(\partial/\partial \epsilon)|_{\epsilon=0}\gamma_\epsilon(t)$ is vertical with respect to the fiberation $\mathcal{P}\to \Sigma$. Then, we~have
\[
\frac{\mathrm{d}}{\mathrm{d} \epsilon}\Big|_{\epsilon=0}
\int _a^b \frac{G^{\mathrm{F}}(\dot \gamma_\epsilon(t), \dot \gamma_\epsilon(t))}{G^{\mathrm{F}}(K, \dot \gamma_\epsilon(t))}\,{\rm d}t=0.
\]
\end{Proposition}
\begin{proof}
By the assumption on $\gamma(t)$, we have
\begin{equation}\label{phi-gamma}
\phi(\dot\gamma(t))=\begin{pmatrix} 0 \\ 0 \\ |\dot x|\end{pmatrix}.
\end{equation}

We may assume that the variation is of the form
\[
\gamma_\epsilon(t)=R_{{\rm e}^{\epsilon X(t)}}\gamma(t)=\gamma(t)\cdot {\rm e}^{\epsilon X(t)},
\]
where
\[
X(t)=\begin{pmatrix}
\hphantom{-}0 & \hphantom{-}X^3(t) & -X^2(t) \\
-X^3(t) & \hphantom{-}0 & \hphantom{-}X^1(t) \\
\hphantom{-}X^2(t) & -X^1(t) & \hphantom{-}0
\end{pmatrix}
\]
is an $\mathfrak{so}(3)$-valued function with $X(a)=X(b)=O$.

For any $X\in\mathfrak{so}(3)$, we denote by $X^\dagger$ the vertical vector field on $\mathcal{P}$ generated by $X$:
\[
X^\dagger_u:=\frac{\partial}{\partial s}\biggl|_{s=0}R_{{\rm e}^{sX}} u, \qquad u\in \mathcal{P}.
\]
Then, we have
\[
(R_{A})_* X^\dagger=\bigl({\rm Ad}\bigl(A^{-1}\bigr)X\bigr)^\dagger
\]
for any $A\in{\rm SO}(3)$ and $X\in\mathfrak{so}(3)$.

The velocity of the curve $\gamma_\epsilon(t)$ is given by
\[
\dot\gamma_\epsilon(t)=(R_{{\rm e}^{\epsilon X(t)}})_* \dot\gamma(t)+\epsilon\dot X(t)^{\dagger}.
\]
Then, by using the right-invariance of $G^{\rm F}$ and the fact that the fibers of $\mathcal{P}$ are totally null, we~compute as
\begin{align*}
G^{\rm F}(\dot\gamma_\epsilon, \dot\gamma_\epsilon)
={}&G^{\rm F} ((R_{{\rm e}^{\epsilon X}})_* \dot\gamma, (R_{{\rm e}^{\epsilon X}})_* \dot\gamma )
 2\epsilon G^{\rm F}\bigl((R_{{\rm e}^{\epsilon X}})_* \dot\gamma,
\dot X^{\dagger}\bigr)+\epsilon^2G^{\rm F}\bigl(\dot X^\dagger, \dot X^\dagger\bigr) \\
={}&G^{\rm F}(\dot\gamma, \dot\gamma)
+2\epsilon G^{\rm F}\bigl(\dot\gamma, \bigl({\rm Ad}\bigl({\rm e}^{\epsilon X}\bigr)\dot X\bigr)^\dagger\bigr).
\end{align*}
Hence, using \eqref{phi-gamma}, we have
\begin{equation*}
\frac{\partial}{\partial \epsilon}\bigg|_{\epsilon=0}G^{\rm F}(\dot\gamma_\epsilon, \dot\gamma_\epsilon)
=2G^{\rm F}\bigl(\dot\gamma, \dot X^\dagger\bigr)
=2|\dot x|\omega^3\bigl(\dot X^\dagger\bigr)=2|\dot x|\dot X^3.
\end{equation*}
On the other hand, since $K=Y^\dagger$ the denominator of the integrand is computed as
\begin{align*}
G^{\rm F}(K, \dot\gamma_\epsilon)
=G^{\rm F}\bigl(Y^\dagger, (R_{{\rm e}^{\epsilon X}})_* \dot\gamma+\epsilon\dot X^{\dagger}\bigr)
=G^{\rm F}\bigl(\bigl({\rm Ad}\bigl({\rm e}^{\epsilon X}\bigr) Y\bigr)^\dagger, \dot\gamma\bigr)
=\omega^3\bigl(\bigl({\rm Ad}\bigl({\rm e}^{\epsilon X}\bigr) Y\bigr)^\dagger\bigr),
\end{align*}
and hence
\[
\frac{\partial}{\partial \epsilon}\bigg|_{\epsilon=0}G^{\rm F}(K, \dot\gamma_\epsilon)
=\omega^3\bigl([X, Y]^\dagger\bigr) =\omega^3\left(\begin{pmatrix}
0 & 0 & -X^1 \\
0 & 0 & -X^2 \\
X^1 & X^2 & \hphantom{-}0
\end{pmatrix}^{\dagger}\right)=0.
\]
Thus, noting that $G^{F}(K, \dot\gamma)=|\dot x|\omega^3\bigl(Y^\dagger\bigr)=|\dot x|$, we have
\[
\frac{\mathrm{d}}{\mathrm{d} \epsilon}\bigg|_{\epsilon=0}
\int _a^b \frac{G^{\rm F}(\dot \gamma_\epsilon(t), \dot \gamma_\epsilon(t))}{G^{\rm F}(K, \dot \gamma_\epsilon(t))}\,{\rm d}t=
2\int_a^b \dot X^3(t) \, {\rm d}t=2\bigl(X^3(b)-X^3(a)\bigr)=0.
\tag*{\qed}
\]
\renewcommand{\qed}{}
\end{proof}

Using this proposition, we prove the following variational characterization of conformal geodesics.

\begin{Theorem}\label{variational-characterization}
Let $x(t)$, $a\le t\le b$, be a regular curve on $(\Sigma, g)$ and $\tilde x(t)=(x(t), \dot x(t)/|\dot x(t)|)$ its lift to $S\Sigma$. Let $F$ be the Kropina metric on $S\Sigma$ associated with a local section $s$ of $\mathcal{P}\to S\Sigma$ defined in a neighborhood of $\tilde x(t)$. Then, $x(t)$ is a conformal geodesic on $(\Sigma, g)$ if and only if it is a critical curve of the functional
\begin{equation}\label{conf-geod-functional}
\mathscr{L}[x(t)]=\int_a^b F\bigl(\tilde x(t), \dot{\tilde x}(t)\bigr)\, {\rm d}t
\biggl(=2\int_a^b \omega^3\bigl(s_*\dot{\tilde x}\bigr)\,{\rm d}t\biggr)
\end{equation}
under the variations fixing $\tilde x(a)$ and $\tilde x(b)$.
\end{Theorem}
\begin{proof}
If $x(t)$ is a conformal geodesic, $\tilde x(t)$ is a chain on $S\Sigma$ by
Theorem~\ref{chain-conf-geod2}. Thus, $\tilde x(t)$ is a~critical curve for the Kropina length and hence $x(t)$ is a critical curve of $\mathscr{L}$ for the variations fixing $\tilde x(a)$ and $\tilde x(b)$.

Conversely, suppose $x(t)$ is a critical curve of $\mathscr{L}$. We will show that the Kropina length functional $L$ is stationary under variations of $\tilde x(t)$. Let $y_\epsilon(t)$ be an arbitrary variation of $\tilde x(t)$ and let $V(t)=(\partial/\partial \epsilon)|_{\epsilon=0}y_\epsilon(t)$ be the variation vector field. We take a variation $x_\epsilon(t)$ of $x(t)$ which has the variation vector field $\pi_* V(t)$, where $\pi\colon S\Sigma\to \Sigma$ is the projection. Then, we can write as
\[
V(t)=V_1(t)+V_2(t),
\]
where $V_1(t):=(\partial/\partial \epsilon)|_{\epsilon=0}\tilde x_\epsilon (t)$ is the variation vector field of the lift $\tilde x_\epsilon (t):=(x_\epsilon(t), \dot x_\epsilon(t)/|\dot x_\epsilon (t)|)$ and $V_2(t)\in \operatorname{Ker} \pi_*$ is a vertical vector field. By the assumption, we have $L'(V_1)=\mathscr{L}'(\pi_*V)=0$. Since we also have $L'(V_2)=0$ by Proposition~\ref{vertical-variation}, we obtain $L'(V)=0$.
\end{proof}

\subsection{The total torsion functional}
We will discuss the relation between the functional $\mathscr{L}$ and the total torsion functional, which is defined as follows.

Let $x(t)$, $a\le t\le b$, be a regular curve on $(\Sigma, g)$. We fix oriented orthonormal basis $A=(e_1, e_2, \dot x(a)/|\dot x(a)|)\in \mathcal{P}_{x(a)}$ and
$B=\bigl(e'_1, e'_2, \dot x(b)/|\dot x(b)|\bigr)\in \mathcal{P}_{x(b)}$ at the endpoints.
Then, we have the associated orientation preserving orthogonal map
\[
h\colon\ \dot x(a)^\perp\longrightarrow \dot x(b)^\perp,
\]
satisfying $h(e_i)=e'_i$. This is called the {\it monodromy map} in~\cite{CKPP}. We also write
\[
P\colon\ \dot x(a)^\perp\longrightarrow \dot x(b)^\perp
\]
for the parallel transport along $x(t)$ with respect to the normal connection $\nabla^\perp$. Then, the {\it total torsion $($functional$)$} is defined by
\[
\mathscr{T}[x(t)]:=h^{-1}\circ P\in \mathrm{SO}\bigl(\dot x(a)^\perp\bigr)\cong \mathbb{R}/2\pi\mathbb{Z}.
\]
Note that the definition of $\mathscr{T}$ depends on the choice of orthonormal basis $A$, $B$ at the endpoints.
\begin{Proposition}[{\cite[Lemma 1]{CKPP}}]\label{torsion-integral-proposition}
Let $T(t):=\dot x(t)/|\dot x(t)|$ be the unit velocity vector field, and
let $\nu(t)\in \dot x(t)^\perp$ be a unit normal vector field such that $h(\nu(a))=\nu(b)$. Then, we have
\begin{equation}\label{torsion-integral}
\mathscr{T}[x(t)]\equiv -\int_a^b g(\nabla_{\dot x}\nu, T\times \nu)\,{\rm d}t\mod 2\pi\mathbb{Z}.
\end{equation}
\end{Proposition}
\begin{proof}
Let $\xi(t)$ be a unit normal vector field along $x(t)$ obtained by the parallel transport of~$\nu(a)$ with respect to $\nabla^\perp$. Using the $\mathrm{U}(1)$-action on $\dot x(t)^\perp$, we can write as
\[
\xi(t)= {\rm e}^{{\rm i}\alpha(t)}\cdot \nu(t)
\]
with a real function $\alpha(t)$ satisfying $\alpha(a)=0$. Then, we have $\mathscr{T}[x(t)]\equiv \alpha(b)$ mod $2\pi\mathbb{Z}$ and
\[
g(\nabla_{\dot x}\nu, T\times \nu)=g\bigl(\nabla^\perp_{\dot x}\nu, T\times \nu\bigr)=-\dot \alpha.
\]
Hence, we obtain \eqref{torsion-integral}.
\end{proof}

\begin{Remark}
We added the assumption $h(\nu(a))=\nu(b)$, which is missed in {\cite[Lemma 1]{CKPP}}. We also modified the formula by putting the negative sign.
\end{Remark}

The formula \eqref{torsion-integral} implies that the functional $\mathscr{T}$ is conformally invariant.

\begin{Proposition}\label{conf-invariance-torsion}
Let $\mathscr{T}[x(t)]$ be the total torsion associated with orthonormal basis $A\in \mathcal{P}_{x(a)}$ and $B\in \mathcal{P}_{x(b)}$. Let \smash{$\wh g={\rm e}^{2\Upsilon}g$} be a rescaled metric and let \smash{$\wh{\mathscr{T}}[x(t)]$} be the total torsion for \smash{$\wh g$} associated with the orthonormal basis \smash{${\rm e}^{-\Upsilon}A$} and \smash{${\rm e}^{-\Upsilon}B$}. Then, we have
\[
\mathscr{T}[x(t)]\equiv\wh{\mathscr{T}}[x(t)]\mod 2\pi\mathbb{Z}.
\]
\end{Proposition}
\begin{proof}
The unit velocity vector field of $x(t)$ with respect to $\wh g$ is given by $\wh T:={\rm e}^{-\Upsilon}T$, and we can use $\wh \nu:={\rm e}^{-\Upsilon}\nu$ as a unit normal vector field. Then, we have \smash{$\wh T\,\wh{\times}\,\wh \nu={\rm e}^{-\Upsilon}T\times \nu$}, where $\wh{\times}$ denotes the cross product with respect to $\wh g$. Moreover, the conformal transformation formula of the Levi-Civita connection gives
\begin{align*}
\wh\nabla_{\dot x}\wh\nu ={\rm e}^{-\Upsilon}
 (-(\partial_t \Upsilon) \nu+\nabla_{\dot x}\nu+(\partial_t\Upsilon) \nu +(\nu\cdot \Upsilon) \dot x )
 ={\rm e}^{-\Upsilon} (\nabla_{\dot x}\nu+(\nu\cdot \Upsilon) \dot x ).
\end{align*}
Hence, we obtain
\[
\wh g\bigl(\wh\nabla_{\dot x}\wh\nu, \wh T\,\wh{\times}\,\wh \nu\bigr)
=g\bigl(\nabla_{\dot x}\nu, T\times \nu\bigr).
\tag*{\qed}
\]
\renewcommand{\qed}{}
\end{proof}

The functionals $\mathscr{L}$ and $\mathscr{T}$ are related as follows.

\begin{Proposition}\label{F-T}
Let $\mathscr{L}$ be the functional given by \eqref{conf-geod-functional}. Let $\mathscr{T}$ be the total torsion functional associated to the orthonormal basis $s(\tilde x(a))$, $s(\tilde x(b))$ at the endpoints. Then, we have
\[
\mathscr{L}[x(t)]\equiv 2\mathscr{T}[x(t)]\mod 2\pi\mathbb{Z}.
\]
\end{Proposition}
\begin{proof}
If we write $s(\tilde x(t))=(\nu(t), \nu'(t), T(t))$, then $\nu(t)$ is a unit normal vector field satisfying $h(\nu(a))=\nu(b)$ and we have
\[
-g(\nabla_{\dot x}\nu, T\times \nu)=-g\bigl(\nabla_{\dot x}\nu, \nu'\bigr)=\omega^3\bigl(s_*\dot{\tilde x}(t)\bigr).
\]
Hence we obtain the equation by \eqref{conf-geod-functional} and \eqref{torsion-integral}.
\end{proof}

It follows from this proposition and Theorem~\ref{variational-characterization} that a conformal geodesic $x(t)$, $a\le x\le b$, is characterized as a critical curve for the total torsion functional with fixed $A$, $B$ under the variations fixing $\tilde x(a)$, $\tilde x(b)$.

When the curve $x(t)$ is parametrized by an arclength parameter and has non-vanishing geodesic curvature ($\nabla_{\dot x}\dot x\neq 0$), we can use
\[
\nu=\frac{\nabla_{\dot x}\dot x}{|\nabla_{\dot x}\dot x|}
\]
as a unit normal vector field. In this case, the function
\[
\tau(t):=g(\nabla_{\dot x}\nu, T\times \nu)
\]
is called the {\it torsion (function)} of $x(t)$. It is well-known that the torsion is represented as
\[
\tau(t)=\frac{\det (\dot x, \nabla_{\dot x}\dot x, \nabla_{\dot x}\nabla_{\dot x}\dot x)}{|\dot x\times \nabla_{\dot x}\dot x|^2}
\]
for an arbitrary parametrization. It may not hold that $h(\nu(a))=\nu(b)$ for $A=s(\tilde x(a))$, ${B=s(\tilde x(b))}$, but we can change $B$ so that it holds, and this only causes a difference in $\mathscr{T}$ by a~constant $c\in\mathbb{R}/2\pi\mathbb{Z}$. Hence, we have
\[
\mathscr{T}[x(t)]\equiv-\int_a^b \frac{\det (\dot x, \nabla_{\dot x}\dot x, \nabla_{\dot x}\nabla_{\dot x}\dot x)}{|\dot x\times \nabla_{\dot x}\dot x|^2}|\dot x|\,{\rm d}t+c\mod 2\pi\mathbb{Z}
\]
for an arbitrary parametrization. Thus we complete the proof of Theorem~\ref{conf-geod-total-torsion}.

Chern--Kn\"oppel--Pedit--Pinkall~\cite{CKPP} calculated the variation of the total torsion functional.
We define $*R\in T^*\Sigma\otimes \mathfrak{so}(T\Sigma)$ by applying the Hodge star operator $*\colon \wedge^2 T^*\Sigma\to T^*\Sigma$ to the curvature tensor $R\in \wedge^2 T^*\Sigma\otimes \mathfrak{so}(T\Sigma)$. In the index notation, we have
\[
(*R)_i{}^k{}_l=\frac{1}{2}\varepsilon_i{}^{pq}R_{pq}{}^k{}_l.
\]

\begin{Theorem}[{\cite[Theorem 2]{CKPP}}]\label{variation-T}
Let $x(t)$, $a\le t \le b$, be a regular curve on $(\Sigma, g)$. Assume~$x_\epsilon (t)$
is a~variation of $x(t)$ which fixes $(x(a), \dot x(a)/|\dot x(a)|)$ and $(x(b), \dot x(b)/|\dot x(b)|)$, and $V(t)=(\partial/\partial \epsilon)|_{\epsilon=0}x_\epsilon(t)$ its variation vector field. If we define the total torsion $\mathscr{T}[x_\epsilon(t)]$ by using common orthonormal basis at the endpoints, we have
\[
\frac{\mathrm{d}}{\mathrm{d} \epsilon}\bigg|_{\epsilon=0}\mathscr{T}[x_\epsilon(t)]=
\int_a^b g\bigl(T\times \nabla_T\nabla_T T+(*R)(T)T, V\bigr)|\dot x|\,{\rm d}t,
\]
where $T=\dot x/|\dot x|$.
\end{Theorem}

We obtain an alternative proof of Theorem~\ref{conf-geod-total-torsion} from this formula.
By \eqref{curvature-tensor},
we have
\[
(*R)_{mkl}=\frac{1}{2}\varepsilon_m{}^{ij}R_{ijkl}
=\frac{1}{2}\bigl(\varepsilon_m{}^i{}_l P_{ik}-\varepsilon_{ml}{}^j P_{jk}
+\varepsilon_{mk}{}^j P_{jl}-\varepsilon_m{}^i{}_k P_{il}\bigr)
\]
and hence
\begin{align*}
((*R)(T)T)^k =(*R)_m{}^k{}_lT^m T^l
 =\frac{1}{2}\bigl(\varepsilon_m{}^k{}_jP^j{}_l T^m T^l-\varepsilon_{mi}{}^kP^i{}_l T^mT^l\bigr)
 =-(T\times P(T))^k.
\end{align*}
Thus, we have
\[
T\times \nabla_T\nabla_T T+(*R)(T)T=
T\times (\nabla_T \nabla_T T-P(T)),
\]
which implies that the variation of the total torsion vanishes if and only if $x(t)$ is a conformal geodesic.

\subsection*{Acknowledgements}
The author thanks the anonymous referees for their comments.
This work was partially supported by JSPS KAKENHI Grant Number 22K13922.

\pdfbookmark[1]{References}{ref}
\LastPageEnding

\end{document}